\documentclass{imanum-eh}
\usepackage[utf8]{inputenc}
\usepackage{amsmath,amssymb,amsthm}
\usepackage{graphicx,graphics,color}
\usepackage{fancyhdr}
\usepackage{url}

\usepackage{caption}
\usepackage{tikz,subcaption}

\newtheorem{Condition}{Condition}[section]
\newcommand{\MU}{\mu}
\newcommand{\KAPPA}{\kappa}
\newcommand{\BETA}{\beta}

\newcommand{\M}{{\bf M}} % mass lumped mass matrix

\newcommand\bfb{{\mathbf b}}

\newcommand\bff{{\mathbf f}}

\newcommand\bfu{{\mathbf u}}

\newcommand\bfy{{\mathbf y}}

\newcommand\bfA{{\mathbf A}}

\newcommand\bfE{{\mathbf E}}

\newcommand\bfI{{\mathbf I}}
\newcommand\bfL{{\mathbf L}}
\newcommand\bfM{{\mathbf M}}
\newcommand\bfS{{\mathbf S}}
\newcommand\bfT{{\mathbf T}}
\newcommand\bfX{{\mathbf X}}

\newcommand\bfzero{{\mathbf 0}}

\newcommand\for{\quad\hbox{ for}\quad }

\renewcommand\d{\hbox{\rm d}}
\newcommand\dt{\tau}

\newcommand\calE{{\cal E}}
\newcommand\calT{{\mathcal T}}

\newcommand{\bfv}{\boldsymbol{v}}

\newcommand{\andquad}{\qquad \textrm{ and } \qquad}

\def \d {\mathrm{d}}
\newcommand{\diff}{\frac{\d}{\d t}}

\newcommand{\Ga}{\Gamma}

\newcommand{\Half}{\frac{1}{2}}

\newcommand{\laplace}{\Delta}

\newcommand{\lm}{{\rm LM}}

\newcommand{\nb}{\nabla}
\newcommand{\Om}{\Omega}

\newcommand{\pa}{\partial}
\newcommand{\R}{\mathbb{R}}

\newcommand{\resp}{respectively}

\newcommand{\st}{such that}
\def \t {(t) }
\def \to {\rightarrow}
\newcommand{\vphi}{\varphi}
\begin{document}
\title{Numerical analysis of parabolic problems \\
with dynamic boundary conditions}
\shorttitle{Parabolic problems 
with dynamic boundary conditions}
\author{Bal\'azs Kov\'acs and Christian Lubich }
\author{%
{\sc
Bal\'azs Kov\'acs\thanks{Email: koboaet@cs.elte.hu}}\\[2pt]
MTA-ELTE NumNet Research Group, P\'azm\'any P. s\'et\'any 1/C., \\
1117 Budapest, Hungary\\[6pt]
{\sc and}\\[6pt]
{\sc Christian Lubich\thanks{Email: Lubich@na.uni-tuebingen.de}}\\[2pt]
Mathematisches Institut, Universit\"at T\"ubingen, Auf der Morgenstelle,\\
72076 T\"ubingen, Germany
}
\shortauthorlist{B. Kov\'acs and Ch. Lubich}
%\date{}

\maketitle

\begin{abstract}
% Body of abstract:
{Space and time discretisations of parabolic differential equations with dynamic boundary conditions are studied in a weak formulation that fits into the standard abstract formulation of parabolic problems, just that the 
usual $L_2(\Omega)$ inner product is replaced by an
$L_2(\Omega)\oplus L_2(\partial\Omega)$ inner product. The class of parabolic equations considered includes linear problems with time- and space-dependent coefficients and semi-linear  problems such as reaction-diffusion on a surface coupled to diffusion in the bulk. 
The spatial discretisation by finite elements  is studied in the proposed framework, with particular attention  to the error analysis of the Ritz map for the elliptic bilinear form in relation to the inner product, both of which contain boundary integrals. The error analysis is done for both polygonal and smooth domains. We further consider mass lumping, which  enables us to use exponential integrators and bulk-surface splitting for time integration.}
% Keywords:
{dynamic boundary condition, Wentzell boundary condition, bulk-surface coupling, gradient flow, bulk and surface finite elements, Ritz projection, time discretisation, stability, a-priori error bounds}

\end{abstract}

\section{Introduction}
We are interested in the numerical solution of parabolic initial-boundary value problems with dynamic boundary conditions. Prototypes for this class of problems are the {\it heat equation with Wentzell boundary conditions}
\begin{equation} \label{heat-dynbc-wentzell}
%\label{eq: PDE problem}
    \begin{cases}
        \begin{alignedat}{4}
            \pa_t u &= \laplace u &\qquad & \textrm{in} \quad \Om\\
            \MU\, \pa_t u &= -\KAPPA u - \pa_{\nu} u &\qquad & \textrm{on} \quad  \Gamma,
        \end{alignedat}
    \end{cases}
\end{equation}
set on a bounded, piecewise smooth domain $\Omega\subset\R^d$ (the bulk) with boundary $\Gamma=\partial\Omega$ (referred to as the surface), with a positive coefficient $\MU$ and real $\KAPPA$ and with $\pa_{\nu} u$ denoting the normal derivative of $u$ on $\Gamma$; and {\it diffusion on the surface coupled to diffusion in the bulk},
\begin{equation} \label{heat-dynbc-heat}
%\label{eq: PDE problem}
    \begin{cases}
        \begin{alignedat}{4}
            \pa_t u &= \laplace u &\qquad & \textrm{in} \quad \Om\\
            \MU\, \pa_t u &= \BETA\, \Delta_\Gamma u - \pa_{\nu} u &\qquad & \textrm{on} \quad  \Gamma,
        \end{alignedat}
    \end{cases}
\end{equation}
with  positive coefficients $\MU$ and $\BETA$ and with the Laplace--Beltrami operator $\Delta_\Gamma$. We will study numerical methods for such equations as well as for inhomogeneous, non-autonomous and semi-linear  variants.

It turns out that such problems admit a weak formulation that fits into the standard abstract framework of parabolic problems. % see for instance \cite{DautrayLions}
The difference to the heat equation with homogeneous Neumann boundary conditions is mainly that the role of the Hilbert spaces $L_2(\Omega)$ and $H^1(\Omega)$ in the Neumann problem is taken by other Hilbert spaces, for the above problems by $L_2(\Omega)\oplus L_2(\Gamma)$ and an appropriate subspace of $H^1(\Omega)$, respectively.

The papers by \cite{CPP} and \cite{cherfils2014numerical} on the finite element discretisation of the Cahn--Hilliard equation on a slab with various dynamic boundary conditions are the only publications on the numerical analysis of parabolic problems with dynamic boundary conditions that we know of.
On the other hand, there has been much recent work on analytical and modelling aspects of such problems; see, e.g., \cite{CavGGM10,coclite2009stability,ColF,engel2005analyticity,Favini2002heat,Gal2008well,Gal2008non,Goldstein2006derivation,goldstein2011cahn,Kenzler2001,Liero,Racke2003cahn,vazquez2011heat}. We are not aware that the abstract framework of the present paper is common in the literature, but related variational settings do appear in some of the references.

In Section~\ref{sect:framework} we discuss the abstract variational framework and how problems such as \eqref{heat-dynbc-wentzell} and \eqref{heat-dynbc-heat} fit in, as well as variants of these problems with space- and time-dependent coefficient functions and semi-linear  problems such as reaction-diffusion on the boundary coupled to diffusion in the bulk, the Allen--Cahn equation with dynamic boundary conditions, and the Cahn--Hilliard equation with various types of dynamic boundary conditions.

In Sections 3 and 4 we study the finite element semi-discretisation in space, which for problems \eqref{heat-dynbc-wentzell} and \eqref{heat-dynbc-heat} leads to a large system of ordinary differential equations for the nodal vector $\bfu(t)$,
$$
\bfM \dot \bfu(t) + \bfA \bfu(t)=\bfzero,
$$
where the positive definite mass matrix $\bfM$ and the positive definite or semi-definite stiffness matrix $\bfA$
differ from those of the heat equation with Neumann boundary conditions only in entries that correspond to pairs of  adjacent boundary nodes.
Much of the standard numerical analysis of the heat equation with Dirichlet or Neumann boundary conditions carries over in a direct way, since the problems with dynamic boundary conditions share the same  abstract variational framework. There are, however, a few issues where care is needed in the extension of the theory:
\begin{itemize}
\item The Ritz projection is based on a different elliptic form that contains boundary integrals. It must be put in relation with a different inner product that also contains boundary integrals.
\item Mass lumping is done for an inner product with boundary terms. It behaves differently for \eqref{heat-dynbc-wentzell} than for \eqref{heat-dynbc-heat}.
\item In the case of non-polygonal domains, the boundary approximation may play a more important role than for pure bulk problems. Its effect on the approximation needs to be studied thoroughly.
\end{itemize}
We first study a class of linear constant-coefficient problems that includes \eqref{heat-dynbc-wentzell} and \eqref{heat-dynbc-heat}, then turn to problems with space- and time-dependent coefficient functions and finally to semi-linear  problems such as the Allen--Cahn equation with dynamic boundary conditions. We discuss finite element space discretisation for the case of polygonal domains in Section 3 and then extend the results to smooth domains in Section 4.

In Section 5 we turn to time discretisation. Known stability and approximation results for standard implicit integrators such as backward difference formulae and Radau IIA implicit Runge--Kutta methods extend from Dirichlet or Neumann boundary conditions to dynamic boundary conditions without much ado, again thanks to the common abstract framework. This framework also makes it obvious how to apply exponential integrators to the problem class studied here, which is not immediately evident from the strong formulation \eqref{heat-dynbc-wentzell} or \eqref{heat-dynbc-heat}. We further consider two classes of bulk-surface splitting integrators,  force splitting and  component splitting, where in both cases differential equations corresponding to the interior domain and to the boundary are solved separately in an alternating way.

\section{Variational formulation of parabolic problems with dynamic boundary conditions}
\label{sect:framework}

\subsection{Linear problems with time-independent operators}
\label{subsec:lin-const}

\subsubsection{Abstract setting}
We recall the usual weak formulation of abstract linear parabolic problems: Given  are two Hilbert spaces $V$, with norm $\|\cdot\|$, and $H$, with norm $|\cdot|$ corresponding to the inner product $(\cdot,\cdot)$, such that $V$ is densely and continuously embedded in $H$. On $V$ we consider a continuous bilinear form $a(\cdot,\cdot)$ that satisfies a G\aa rding inequality: there exist $\alpha>0$ and real $c$ such that
$$
a(v,v) \ge \alpha \|v\|^2 - c |v|^2 \qquad \forall\, v\in V.
$$
On a time interval $0\le t \le T$, for given initial data $u_0\in H$ and an inhomogeneity $f\in L_2(0,T;H)$, the abstract parabolic initial value problem then reads: Find $u\in C([0,T],H) \cap L_2(0,T;V)$ such that (with $\dot u=\d u/\d t$)
\begin{equation}
\label{P}
%(P) \qquad\quad   %\begin{cases}
        \begin{alignedat}{1}
            (\dot u(t),v) + a(u(t),v) &= (f(t),v) \qquad \forall \,v\in V \qquad (0<t\le T)\\
            u(0) &= u_0\,.
        \end{alignedat}
   % \end{cases}
\end{equation}
We note that (for $f=0$) this can be viewed as the $H$-gradient flow, $(\dot u,v)=-E'(u)v$ for all $v\in V$, for the quadratic energy functional
$$
E(v) = \tfrac12 \, a(v,v), \qquad v \in V.
$$
The well-posedness of this abstract problem is well known; see, e.g., \cite{DautrayLions, Kato}. An {\it a priori} estimate of the solution is obtained by the familiar energy technique:
test with $v=u(t)$, note $(\dot u(t),u(t))=\frac12 \frac d{dt}|u(t)|^2$ and integrate from 0 to $t$ to obtain
$$
\tfrac12 |u(t)|^2 - \tfrac12 |u_0|^2+ 2\int_0^t E(u(s))\, ds = \int_0^t (f(s),u(s)) \, \d s.
$$
Using here the bound $|(f(s),u(s))| \le \| f(s) \|_* \, \| u(s) \|$ with the dual norm
$\|\varphi\|_* = \sup_{\| v \|=1} |(\varphi,v)|$ and estimating this further as $|(f(s),u(s))|\le \frac1{2\alpha} \| f(s) \|_*^2 + \frac\alpha 2 \|u(s)\|^2$ and using the G\aa rding inequality, one arrives at
$$
|u(t)|^2 +\alpha  \int_0^t \| u(s) \|^2 \, \d s \le |u_0|^2 + \frac 1\alpha\int_0^t  \| f(s) \|_*^2\, \d s + 2c\int_0^t |u(s)|^2 \, \d s,
$$
and Gronwall's inequality then yields
\begin{equation}\label{en-est}
|u(t)|^2 + \alpha  \int_0^t \| u(s) \|^2 \, \d s \le
e^{2ct} \Bigl( |u_0|^2 + \frac1\alpha \int_0^t  \| f(s) \|_*^2\, \d s \Bigr).
\end{equation}

The standard example is $V=H^1_0(\Omega)$ and $H=L_2(\Omega)$ with the Dirichlet energy functional $E(v)=\frac12 \| \nabla v\|_{L_2(\Omega)}^2$, in which case \eqref{P} yields the weak formulation of the heat equation on the domain $\Omega$ with homogeneous Dirichlet boundary conditions.

\subsubsection{Weak formulation of the heat equation with dynamic boundary conditions}
As it turns out, the weak formulation of the heat equation with dynamic boundary conditions  \eqref{heat-dynbc-wentzell} and \eqref{heat-dynbc-heat} fits into the same abstract framework with the particular choice of Hilbert space
$V=\{ v\in H^1(\Omega)\,:\, \sqrt{\BETA}\, \gamma v\in H^1(\Gamma)\}$,  where $\gamma v$ denotes the trace of $v$ on the boundary $\Gamma=\partial\Omega$, and with the bilinear form on $V$
\begin{equation}\label{a-heat-dynbc}
a(u,v) = \int_\Omega \nabla u \cdot \nabla v \, \d x + \KAPPA \int_\Gamma (\gamma u)(\gamma v)\, \d\sigma
+\BETA \int_\Gamma \nabla_\Gamma(\gamma u)\cdot\nabla_\Gamma(\gamma v)\, \d\sigma \qquad (\KAPPA\in\R,\,\BETA\ge 0),
\end{equation}
where $\nabla_\Gamma(\gamma v)=(I-\nu\nu^T)\gamma(\nabla v)$ is the tangential gradient on $\Gamma$ ($\nu$ denotes the unit normal on $\Gamma$), which is known to depend only on the trace $\gamma v$ on $\Gamma$. For brevity we will write $\nabla_\Gamma v$ instead of $\nabla_\Gamma(\gamma v)$ in the following. We have $\BETA=0$ for
\eqref{heat-dynbc-wentzell} and $\KAPPA=0$ for \eqref{heat-dynbc-heat}.

The Hilbert space $H$ is the completion of $V$ with respect to the norm induced by the inner product
\begin{equation}\label{b-heat-dynbc}
(u,v) = \int_\Omega uv\, \d x + \MU  \int_\Gamma (\gamma u)(\gamma v)\, \d\sigma \qquad (\MU> 0),
\end{equation}
so that $H$ is (isomorphic to) $L_2(\Omega)\oplus L_2(\Gamma)$.

The corresponding energy functional is then
$$
E(v) = \tfrac12 \| \nabla v \|_{L_2(\Omega)}^2 + \tfrac12 \KAPPA\,\| \gamma v \|_{L_2(\Gamma)}^2 +
\tfrac12 \BETA\,\| \nabla_\Gamma v \|_{L_2(\Gamma)}^2, \quad\hbox{and}\quad |v|^2 =  \|  v \|_{L_2(\Omega)}^2 + \MU \| \gamma v \|_{L_2(\Gamma)}^2.
$$
We thus obtain the energy estimate, here stated for the weak solution $u(t)=u(\cdot,t)$ of \eqref{P} with $f=0$ for simplicity,
\begin{eqnarray*}
&&\| u(t) \|_{L_2(\Omega)}^2 + \MU \| \gamma u(t) \|_{L_2(\Gamma)}^2
+ \int_0^t \Bigl(  \| \nabla u(s) \|_{L_2(\Omega)}^2 + \KAPPA\,\| \gamma u(s) \|_{L_2(\Gamma)}^2+ \BETA\,\| \nabla_\Gamma u(s) \|_{L_2(\Gamma)}^2\Bigr)\d s
\\
&&\hspace{9cm}\le \| u(0) \|_{L_2(\Omega)}^2 + \MU \| \gamma u(0) \|_{L_2(\Gamma)}^2.
\end{eqnarray*}
The relationship with the strong formulation of the heat equation with dynamic boundary conditions \eqref{heat-dynbc-wentzell} and \eqref{heat-dynbc-heat} is given by the following result.

\begin{lemma} Every classical solution $u\in C^2(\bar\Omega\times[0,T])$ of the heat equation with dynamic boundary conditions
\begin{equation} \label{heat-dynbc}
%\label{eq: PDE problem}
    \begin{cases}
        \begin{alignedat}{4}
            \pa_t u &= \laplace u &\qquad & \textrm{in} \quad \Om\\
            \MU\, \pa_t u &= -\KAPPA u + \BETA \,\Delta_\Gamma u- \pa_{\nu} u &\qquad & \textrm{on} \quad  \Gamma
        \end{alignedat}
    \end{cases}
\end{equation}
with initial data $u_0$ is a solution of the weak formulation \eqref{P} with the  bilinear forms $a(\cdot,\cdot)$ and $(\cdot,\cdot)$ of \eqref{a-heat-dynbc}--\eqref{b-heat-dynbc}. Conversely, if the solution $u$ of \eqref{P}  with  $a(\cdot,\cdot)$ and $(\cdot,\cdot)$ of \eqref{a-heat-dynbc}--\eqref{b-heat-dynbc} is sufficiently regular, then $u$ is a solution of the strong formulation \eqref{heat-dynbc}.
\end{lemma}

The proof is almost identical to the proof of the corresponding result for the heat equation with Neumann boundary conditions. It is again based on Green's formula in the domain and on the boundary, and on the fundamental lemma of variational calculus. The proof is therefore omitted.

\subsection{Linear problems with time-varying operators}

\subsubsection{Abstract setting}
Consider again Hilbert spaces $V$ and $H$ such that $V$ is continuously and densely embedded in $H$. On $V$ we consider uniformly equivalent time-dependent inner product norms $\|\cdot\|_t$, and on $H$ a time-dependent family of inner products $m(t;\cdot,\cdot):H\times H\to\R$ that induce uniformly equivalent norms $|w|_t^2 = m(t;w,w)$ for $w\in H$. We assume  a bounded partial derivative of $m$ with respect to $t$:
\begin{equation}\label{m-dot-bound}
\Bigl| \frac {\partial m}{\partial t} (t;u,v) \Bigr|\le M_0'\, |u|_t\, |v|_t \qquad
\forall\, u,v\in V \qquad (0\le t \le T).
\end{equation}

%
%\hfill[a conflicting notation $\mu$ is now a mass parameter, changed to $\RHO$]
%
On $V$ we consider a time-dependent family of  bilinear forms $a(t;\cdot,\cdot):V\times V\to \R$ ($0\le t \le T$) that satisfy a uniform G\aa rding inequality
$$
a(t;v,v) \ge \alpha \|v\|_t^2 - c |v|_t^2 \qquad \forall\, v\in V \qquad (0\le t \le T)
$$
and are uniformly bounded,
$$
|a(t;u,v)| \le M_1  \| u \|_t \, \| v \|_t \qquad  \forall\, u,v\in V \qquad (0\le t \le T).
$$
We further assume
\begin{equation}\label{a-dot-bound}
\Bigl| \frac {\partial a}{\partial t} (t;u,v) \Bigr|\le M_1'\, \|u\|_t\, \|v\|_t \qquad
\forall\, u,v\in V \qquad (0\le t \le T).
\end{equation}
The non-autonomous version of the abstract parabolic initial value problem is then to find $u\in C([0,T],H)\cap L_2(0,T;V)$ such that
\begin{equation}
\label{P-t}
%(P) \qquad\quad   %\begin{cases}
        \begin{alignedat}{1}
            m(t;\dot u(t),v) + a(t;u(t),v) &= m(t;f(t),v) \qquad \forall \,v\in V \qquad (0<t\le T)\\
            u(0) &= u_0\,.
        \end{alignedat}
   % \end{cases}
\end{equation}
Here again we obtain an {\it a priori} bound using the energy technique: by the same arguments as in the time-invariant situation and using the above bounds we obtain
\begin{equation}\label{en-est-t}
|u(t)|_t^2 + \alpha \int_0^t \| u(s) \|_s^2 \, \d s \le e^{2 (c+M_0') t}
\Bigl( |u_0|_0^2 + \frac1\alpha \int_0^t  \| f(s) \|_{*,s}^2\, \d s \Bigr),
\end{equation}
with the time-dependent dual norm $\| \varphi \|_{*,t}= \sup_{\| v \|_t=1} |m(t; \varphi, v)|$.

\subsubsection{Heat equation with non-autonomous dynamic boundary conditions}

We consider the weak formulation of the problem
\begin{equation} \label{heat-dynbc-t}
%\label{eq: PDE problem}
    \begin{cases}\ \
        \begin{alignedat}{4}
            \pa_t u &= \laplace u &\qquad & \textrm{in} \quad \Om\\
            \MU(x,t) \pa_t u(x,t) &= -\KAPPA(x,t) u(x,t) +\nabla_\Gamma \cdot \BETA(x,t) \nabla_\Gamma u(x,t) - \pa_{\nu} u(x,t) &\qquad & \textrm{on} \quad  \Gamma=\pa\Om
        \end{alignedat}
    \end{cases}
\end{equation}
with real-valued coefficient functions $\MU,\KAPPA,\BETA$ on $\Gamma\times[0,T]$ such that $\MU,\KAPPA,\BETA:[0,T]\to L_\infty(\R)$ are continuously differentiable, $\MU$ has a strictly positive lower bound, and $\BETA$ either  has a strictly positive lower bound or vanishes identically.
This fits into the above framework for
$V=\{v \in H^1(\Omega)\,:\, \sqrt{\beta} \,\nabla_\Gamma(\gamma v) \in L_2(\Gamma)\}$
with
\begin{equation}\label{a-heat-dynbc-t}
a(t;u,v) = \int_\Omega \nabla u \cdot \nabla v \, \d x +  \int_\Gamma \KAPPA(\cdot,t) \,(\gamma u)(\gamma v)\, \d\sigma +
\int_\Gamma \BETA(\cdot,t) \,\nabla_\Gamma u\cdot\nabla_\Gamma v\, \d\sigma
\end{equation}
and  $H=L_2(\Omega)\oplus L_2(\Gamma)$ with
\begin{equation}\label{b-heat-dynbc-t}
m(t;u,v) = \int_\Omega uv\, \d x +   \int_\Gamma \MU(\cdot,t) \,(\gamma u)(\gamma v)\, \d\sigma.
\end{equation}
The framework applies equally when $\MU(\cdot,t)$ is piecewise continuous on $\Gamma$ and  has a positive lower bound on a subset $\Gamma_+\subset \Gamma$ and is zero on the complementary part $\Gamma_0=\Gamma\setminus\Gamma_+$. In this case,
$H=L_2(\Omega)\oplus L_2(\Gamma_+)$. Similarly, $\BETA$ may be allowed to have a positive lower bound on a time-independent subset of $\Gamma$ and to vanish on the complement.

\subsection{Some nonlinear examples}\label{subsec:nonlinear-examples}

We present three examples of nonlinear parabolic equations with dynamic boundary conditions. They can all be cast in the following abstract form of a semi-linear  parabolic problem on suitable spaces $V$ and $H$: Find $u\in C([0,T],H) \cap L_2(0,T;V)$ such that (with $\dot u=du/dt$)
\begin{equation}
\label{P-nl}
%(P) \qquad\quad   %\begin{cases}
        \begin{alignedat}{1}
            (\dot u(t),v) + a(u(t),v) &= (f(u(t)),v) \qquad \forall \,v\in V \qquad (0<t\le T)\\
            u(0) &= u_0\,,
        \end{alignedat}
   % \end{cases}
\end{equation}
where $f:V\to H$ is a sufficiently regular nonlinearity. In our examples the bilinear form $a(\cdot,\cdot)$ on $V$ and the inner product $(\cdot,\cdot)$ on $H$ contain boundary terms. The bilinear form $a(\cdot,\cdot)$ is of the form \eqref{a-heat-dynbc} with $\BETA>0$. In the first two examples the inner product $(\cdot,\cdot)$ is of the form \eqref{b-heat-dynbc} on $H=L_2(\Omega)\oplus L_2(\Gamma)$, and in the third example on $H=H^{-1}(\Omega)\oplus L_2(\Gamma)$ or $H=H^{-1}(\Omega)\oplus H^{-1}(\Gamma)$.

\subsubsection{Reaction-diffusion on a surface coupled to diffusion in the bulk}
We consider a reaction-diffusion equation on the boundary
$$
\MU \,\partial_t \psi = \BETA \,\Delta_\Gamma \psi + f(\psi) \quad\hbox{ on $\Gamma$},
$$
with $\MU,\BETA>0$ %, with the Laplace-Beltrami operator $\Delta_\Gamma$ 
and a nonlinear pointwise reaction term
$f(\phi)(x)=f(\phi(x))$ (for a smooth and bounded function $f:\R\to\R$, say),
and  we further consider diffusion in the bulk,
$$
\partial_t u = \Delta u \quad\hbox{ in $\Omega$}.
$$
We couple these equations subject to the constraint $\psi=\gamma u$. We obtain the following weak formulation: find
$(u,\psi):[0,T] \to H^1(\Omega)\times H^1(\Gamma)$ subject to $\psi=\gamma u$ such that
for all
$ (v,\phi)\in H^1(\Omega)\times H^1(\Gamma)$ with $ \phi=\gamma v$,
$$%\begin{eqnarray*}
 (\dot u, v)_{L_2(\Omega)}+ \MU (\dot\psi, \phi)_{L_2(\Gamma)}= - (\nabla u, \nabla v)_{L_2(\Omega)}
 - \BETA(\nabla_\Gamma \psi, \nabla_\Gamma \phi)_{L_2(\Gamma)}
 + (f(\psi),\phi)_{L_2(\Gamma)}.
$$%\end{eqnarray*}
%with the tangential gradient $\nabla_\Gamma$.
Equivalently, with the bilinear forms
\begin{equation}\label{ab-rd}
\begin{array}{rcl}
 (u,v) &=& (u,v)_{L_2(\Omega)} + \MU (\gamma u,\gamma v)_{L_2(\Gamma)}
 \\[1mm]
a(u,v) &=& (\nabla u, \nabla v)_{L_2(\Omega)}
+ \BETA (\nabla_\Gamma  u, \nabla_\Gamma  v)_{L_2(\Gamma)}
\end{array}
\end{equation}
on the Hilbert spaces $H=L_2(\Omega)\oplus L_2(\Gamma)$ and
$V=\{ v \in H^1(\Omega)\,:\, \gamma v \in H^1(\Gamma)\}$, respectively,  we have
$$
(\dot u, v) + a(u,v) = (f(\gamma u),\gamma v)_{L_2(\Gamma)} \qquad\forall\, v\in V.
$$
This is to be solved for $u\in C([0,T],H)\cap L_2(0,T;V)$ for given initial data $u_0\in H$.
The corresponding strong formulation is
\begin{equation} \label{rd-dynbc}
\begin{array}{rcl}
 \partial_t u &=& \Delta u \qquad\qquad\qquad\qquad\ \hbox{ in $\Omega$}
 \\[1mm]
 \MU \,\partial_t u &=& \BETA\, \Delta_\Gamma u + f(u) - \partial_\nu u \quad\hbox{ on $\Gamma$},
\end{array}
\end{equation}
where the normal derivative $\partial_\nu u$ figures as the Lagrange multiplier corresponding to
the constraint $\psi=\gamma u$.

\subsubsection{Allen--Cahn equation with dynamic boundary conditions}
The following problem is studied in \cite{Gal2008non,Liero,ColF} and further references therein.
Given potentials $W, W_\Gamma:\R\to\R$ such as a double-well potential
$W(u)=(u^2-1)^2$, the Allen--Cahn equation
$$
 \partial_t u = \Delta u -W'(u) \quad\ \hbox{ in $\Omega$}
$$
is considered with a  dynamic boundary condition like in \eqref{rd-dynbc},
$$
\MU \,\partial_t u = \BETA\, \Delta_\Gamma u -W_\Gamma'(u) - \partial_\nu u \quad\hbox{ on $\Gamma$}.
$$
In \cite{Liero} this dynamic boundary condition is derived as a scaling limit in a vanishing boundary layer approximation. This problem fits into the framework \eqref{P-nl} with $(\cdot,\cdot)$ and $a(\cdot,\cdot)$ as in the previous example.   For the energy functional
\begin{equation}\label{EAC}
E(v) = \tfrac12 \| \nabla v \|_{L_2(\Omega)}^2 +
\tfrac12 \BETA\,\| \nabla_\Gamma v \|_{L_2(\Gamma)}^2 +W(v) + W_\Gamma(\gamma v), \qquad v\in V,
\end{equation}
where $V=\{ v \in H^1(\Omega)\,:\, \gamma v \in H^1(\Gamma)\}$ as before, the weak formulation can be viewed as the $L_2(\Omega)\oplus L_2(\Gamma)$ gradient flow  with respect to the weighted inner product \eqref{b-heat-dynbc}:
$$
(\partial_t u,v) = - E'(u)v \qquad \forall \, v\in V.
$$

\subsubsection{Cahn--Hilliard equation with dynamic boundary conditions}
Given potentials $W, W_\Gamma:\R\to\R$, the Cahn--Hilliard equation
$$
\begin{array}{rcl}
\partial_t u &=& \Delta w
\\[1mm]
w &=& \Delta u -W'(u)
\end{array}
\quad\ \hbox{ in $\Omega$}
$$
is considered with the boundary conditions
$$
\begin{array}{rcl}
\MU \,\partial_t u &=& \BETA \, \Delta_\Gamma u  - W_\Gamma'(u) - \partial_\nu u
\\[1mm]
\partial_\nu w &=& 0
\end{array}
\quad\hbox{ on $\Gamma$}
$$
in \cite{CPP,Kenzler2001,Gal2008well,Racke2003cahn} as a  prototype model for the influence of the boundaries on the process of phase separation.

The corresponding weak formulation is the $H^{-1}(\Omega)\oplus L_2(\Gamma)$-gradient flow of the energy functional
\eqref{EAC},  with the inner product
$$
(u,v) = (u,v)_{H^{-1}(\Omega)}+ \MU\, (u,v)_{L_2(\Gamma)},
$$
where $(u,v)_{H^{-1}(\Omega)}=(u, (-\Delta)^{-1}v)_{L_2(\Omega)}$ with $\Delta^{-1}v$ denoting the solution with zero mean of the Poisson equation with homogeneous Neumann boundary conditions for the inhomogeneity $v$ on $\Omega$. This weak formulation is again of the form \eqref{P-nl} with $a(\cdot,\cdot)$ as in \eqref{ab-rd}.

To our knowledge, this problem is the only parabolic problem with  dynamic boundary conditions for which a numerical analysis has been carried out:  in \cite{CPP}  the Elliott--French finite element space discretisation and the implicit Euler time discretisation are studied.
%The corresponding weak formulation is the $H^{-1}(\Omega)\oplus L_2(\Gamma)$-gradient flow of the energy functional
%$$
%E(u) = \int_\Omega \Bigl( \tfrac12 \, |\nabla u|^2 + W(u) \Bigr) \d x +
%\int_\Gamma \Bigl( \BETA \tfrac12 \, |\nabla_\Gamma u|^2
%+ \KAPPA \tfrac12 |u|^2+W_\Gamma(u) \Bigr) \d\sigma,
%$$
%that is,
%$$
%(\partial_t u,v) = - E'(u)v \qquad \forall \, v\in V,
%$$
%with $V= \{ v \in H^1(\Omega)\,:\, \gamma v \in H^1(\Gamma)\}$ and
%$H=H^{-1}(\Omega)\oplus L_2(\Gamma)$ with the inner product
%$$
%(u,v) = (u,v)_{H^{-1}(\Omega)}+ \MU\, (u,v)_{L_2(\Gamma)},
%$$
%where $(u,v)_{H^{-1}(\Omega)}=(u, (-\Delta)^{-1}v)_{L_2(\Omega)}$ with $\Delta^{-1}v$ as the solution with zero average of the Poisson equation with inhomogeneity $v$ on $\Omega$ with Neumann boundary conditions. This weak formulation is again of the form \eqref{P-nl} with $a(\cdot,\cdot)$ as in \eqref{ab-rd}.
%
%To our knowledge, this problem is the only parabolic problem with  dynamic boundary conditions for which a numerical analysis has been carried out:  in \cite{CPP}  the Elliott--French finite element space discretisation and the implicit Euler time discretisation are studied.

The above model can be viewed as coupling Cahn--Hilliard in the bulk and Allen--Cahn on the surface. The Cahn--Hilliard / Cahn--Hilliard coupling corresponding to $H=H^{-1}(\Omega)\oplus H^{-1}(\Gamma)$ is equally of interest; cf.~\cite{goldstein2011cahn}.

\section{Spatial semi-discretisation on polygonal domains}
\label{sect:spacediscretisation}

We study the finite element semi-discretisation in space of parabolic equations with dynamic boundary conditions in the framework of the previous section, first for linear problems with constant coefficients, then for linear problems with space- and time-dependent coefficients, and finally for semi-linear  problems. Mass lumping is also studied. In this section we assume that the domain $\Omega$ is polygonal, so that no boundary approximation is required. 

\subsection{Spatial semi-discretisation of linear problems with constant coefficients}

\subsubsection{Galerkin semi-discretisation}

For a finite dimensional  approximation space $V_h\subset V$, the Galerkin semi-discretisation of the abstract parabolic problem \eqref{P}  determines $u_h:[0,T]\to V_h$ such that
\begin{equation}
\label{heat-dynbc-semi-discrete}
    \begin{alignedat}{1}
        (\dot u_h(t),v_h) + a(u_h(t),v_h) &= (f(t),v_h) \qquad \forall \,v_h\in V_h \qquad (0<t\le T)\\
        (u_h(0),v_h) &= (u_{0} ,v_h) \qquad\ \ \, \forall \,v_h\in V_h .
%u_h(0)=u_{0,h} \in V_h.
    \end{alignedat}
\end{equation}
Representing $u_h(t)$ with respect to a basis $\vphi_1,\vphi_2,\dotsc,\vphi_N$ of $V_h$ as
$$
u_h(t) = \sum_{i=1}^N u_i(t)\vphi_i,
$$
the time-dependent coefficient vector $\bfu(t)=(u_1(t),u_2(t), \dotsc, u_N(t))^T$
then satisfies  the system of linear ordinary differential equations
\begin{equation*}
    \bfM\dot\bfu(t) + \bfA \bfu(t) = \bfb(t),
\end{equation*}
with the symmetric positive definite mass matrix $\bfM$ and the symmetric stiffness matrix $\bfA$ having the entries
\begin{equation}
\label{eq-mass-stiffness-matrix-def}
    \begin{gathered}
        m_{ij}=(\vphi_j,\vphi_i), \qquad %= \int_\Om \vphi_j\vphi_k \d x + \MU \int_\Ga (\gamma \vphi_j)(\gamma \vphi_k) \d \sigma, \\
        a_{ij}= a(\vphi_j,\vphi_i),  \\ %= \int_\Om \nb\vphi_j \cdot \nb\vphi_k \d x + \BETA \int_\Ga \nb_\Ga\vphi_j \cdot \nb_\Ga\vphi_k \d \sigma + \KAPPA \int_\Ga (\gamma \vphi_j)(\gamma \vphi_k) \d \sigma,
    \end{gathered}
    \qquad (i,j=1,2,\dotsc,N),
\end{equation}
and with the load vector $\bfb(t)$ with entries
$$
b_i(t) = (f(t),\vphi_i)  \qquad (i=1,2,\dotsc,N).
$$
There is the semi-discrete energy estimate of the same type as \eqref{en-est}, obtained by the same proof,
\begin{equation}\label{en-est-h}
|u_h(t)|^2 + \alpha  \int_0^t \| u_h(s) \|^2 \, \d s \le
e^{2ct} \Bigl( |u_h(0)|^2 + \frac1\alpha \int_0^t  \| f(s) \|_*^2\, \d s \Bigr).
\end{equation}
%
%$$
%\| u_h(t) \|_{L_2(\Omega)}^2 + \MU \| \gamma u_h(t) \|_{L_2(\Gamma)}^2
%+ \int_0^t \Bigl(  \| \nabla u_h(s) \|_{L_2(\Omega)}^2 + \KAPPA\,\| \gamma u_h(s) \|_{L_2(\Gamma)}^2\Bigr)\d s \le \| u(0) \|_{L_2(\Omega)}^2 + \MU \| \gamma u_h(0) \|_{L_2(\Gamma)}^2.
%$$
%
%\bigskip
%In the non-autonomous case the semi-discrete problem reads as
%\begin{equation*}
%%(P) \qquad\quad   %\begin{cases}
%    \begin{alignedat}{1}
%        m(t;\dot u_h(t),v_h) + a(t;u_h(t),v_h) &= m(t;f(t),v_h) \qquad \forall \,v_h\in V_h \qquad (0<t\le T)\\
%        u_h(0) &= u_{h,0}\,
%    \end{alignedat}
%   % \end{cases}
%\end{equation*}
%The analogous semi-discrete energy estimate is
%$$
%|u_h(t)|_t^2 + \alpha \int_0^t \| u_h(s) \|_s^2 \, \d s \le e^{2 (c+\RHO) t}
%\Bigl( |u_{h,0}|_0^2 + \frac1\alpha \int_0^t  \| f(s) \|_{*,s}^2\, \d s \Bigr).
%$$
%

\subsubsection{Linear finite elements on polygonal domains: first-order error bounds}
\label{subsection-FEM}
We discuss in detail the case of linear finite elements for the heat equation with dynamic boundary conditions \eqref{heat-dynbc}, the extension to higher-order finite elements being straightforward. Using the Ritz projection corresponding to the bilinear form $a(\cdot,\cdot)$ given by \eqref{a-heat-dynbc} we show optimal-order error bounds for the spatially discrete solution. For the ease of presentation we assume that $\KAPPA>0$ in \eqref{a-heat-dynbc} throughout this section, so that $a(\cdot,\cdot)$ is an inner product on the closed subspace $V$ of $H^1(\Omega)$. We then consider the norm $\| v \|^2 = a(v,v)$ on $V$.
The inner product $(\cdot,\cdot)$ on $H=L_2(\Omega)\oplus L_2(\Gamma)$  is given by \eqref{b-heat-dynbc}. It induces the norm $| v |^2 = (v,v)$ on $H$. 

%
%\bigskip
%Let us consider the problem
%\begin{equation}
%\label{eq-general-dynbc-linear-problem}
%    \begin{alignedat}{3}
%        \pa_t u =&\ \laplace u + f_\Om \qquad &\ &\ \textnormal{in } \ \Om \\
%        \MU \pa_t u =&\ \BETA \laplace_\Ga u - \KAPPA u - \pa_\nu u + f_\Ga \qquad &\ &\ \textnormal{in } \ \Ga.
%    \end{alignedat}
%\end{equation}
%For the ease of presentation we assume that $\KAPPA>0$ throughout this section.
%
%The bilinear forms corresponding to the problem are
%\begin{equation}
%\label{eq-form-definitions}
%    \begin{aligned}
%        (u,v) =&\ \int_\Om uv \d x + \MU \int_\Ga (\gamma u)(\gamma v) \d \sigma, \\
%        a(u,v) =&\ \int_\Om \nb u \cdot \nb v \d x + \BETA \int_\Ga \nb_\Ga u \cdot \nb_\Ga v \d \sigma + \KAPPA \int_\Ga (\gamma u)(\gamma v) \d \sigma.
%    \end{aligned}
%\end{equation}

\medskip
We consider a family of quasi-uniform triangulations of the domain $\Om$ parametrised by the maximal meshwidth $h$. The corresponding  finite element space $V_h\subset V$ is spanned by continuous, piecewise linear nodal basis functions $\vphi_1,\vphi_2,\dotsc,\vphi_N$  that are continuous on $\Omega$ and linear on each finite element and, for each node $x_k$,  satisfy $\vphi_j(x_k) = \delta_{jk}$. We note here that the restrictions of the basis functions to the boundary form a basis over the boundary elements.

%The nodal vectors
%%representing the functions in their bases, i.e.\, the matrix formulation for
%$\bfu(t)=(u_1(t),u_2(t), \dotsc, u_N(t))^T$ with $u_k(t)=u_h(x_k,t)$, for which $u_h(.,t) = \sum_{k=1}^N u_k(t)\vphi_k$, then satisfy  the system of linear ordinary differential equations
%\begin{equation}
%\label{eq-matrix-form}
%    \bfM\dot\bfu(t) + \bfA \bfu(t) = \bfb(t).
%\end{equation}
%Here, $\bfM$ is the symmetric positive definite mass matrix given by the inner products of the corresponding basis functions, and $\bfA$ is the symmetric positive semidefinite stiffness matrix. They are given by the following relations
%\begin{equation}
%\label{eq-mass-stiffness-matrix-def}
%    \begin{gathered}
%        \bfM_{kj}=(\vphi_j,\vphi_k), \qquad %= \int_\Om \vphi_j\vphi_k \d x + \MU \int_\Ga (\gamma \vphi_j)(\gamma \vphi_k) \d \sigma, \\
%        \bfA_{kj}= a(\vphi_j,\vphi_k),  \\ %= \int_\Om \nb\vphi_j \cdot \nb\vphi_k \d x + \BETA \int_\Ga \nb_\Ga\vphi_j \cdot \nb_\Ga\vphi_k \d \sigma + \KAPPA \int_\Ga (\gamma \vphi_j)(\gamma \vphi_k) \d \sigma,
%    \end{gathered}
%    \qquad (j,k=1,2,\dotsc,N),
%\end{equation}
%and the load vector $\bfb\t=\int_\Om f_\Om(t)\vphi_k + \int_\Ga f_\Ga(t)\vphi_k$ for $k=1,2,\dotsc,N$.
%
%\noindent$\downarrow$ \hfill[we can also formulate everything with $\MU$]
%
%i.e.
%\begin{align*}
%    \int_\Om f_\Om(t)\vphi_k + \MU \int_\Ga f_\Ga(t)\vphi_k = (f,\vphi_k) \quad \textrm{where } f=(f_\Om,f_\Ga)^T.
%\end{align*}
%$\uparrow$ \hfill[we can also formulate everything with $\MU$]
%

\medskip
The basic tool for proving error bounds is the Ritz projection $R_h:V\to V_h$ with respect to the elliptic bilinear form $a(\cdot,\cdot)$ of \eqref{a-heat-dynbc}, which is the $a$-orthogonal projection defined by
\begin{equation}\label{ritz}
a(R_h u, v_h) = a(u,v_h) \qquad \forall v_h \in V_h.
% I would put here the "detailed" version, as in the proofs and maybe keep it as a definition to highlight it.
\end{equation}

%%%%%%%%%%%%%%%%%%%%%%%%%
% local def
\newcommand{\ur}{u-R_h u}
\newcommand{\ui}{u-I_h u}
%%%%%%%%%%%%%%%%%%%%%%%%%
\begin{lemma}
\label{lemma-Ritz-error}
    The error of the Ritz projection for the elliptic bilinear form \eqref{a-heat-dynbc} satisfies a  first-order bound in the energy norm,
    \begin{align*}
        \|\nb(\ur)\|_{L_2(\Om)}^2 + \BETA \|\nb_\Ga (\ur)\|_{L_2(\Ga)}^2 + &\ \KAPPA \|\gamma(\ur)\|_{L_2(\Ga)}^2\\
        \leq &\  C h^2 \Big( \|u\|_{H^2(\Om)}^2 + \BETA\|\gamma u\|_{H^2(\Ga)}^2 \Big),
    \end{align*}
    where the constant $C$ is independent of $h$  and $u\in H^2(\Omega)$ with $\sqrt{\BETA}\gamma u\in H^2(\Gamma)$.
\end{lemma}

\begin{proof}
With $\|v\|^2=a(v,v)$ we have for the error of the Ritz projection
$$
\| u-R_hu \| = \min_{v_h\in V_h} \| u-v_h \| \le \| u -I_hu \|,
$$
where $I_h$ denotes the piecewise linear finite element interpolation operator.
Using the standard interpolation estimates in the polygonal domain $\Omega$ and on its boundary $\Gamma$, we obtain
\begin{equation}\label{intpol-err}
\begin{array} {rl} \| u -I_hu \|^2 &=\ \|\nb(\ui)\|_{L_2(\Om)}^2 + \BETA \|\nb_\Ga (\ui)\|_{L_2(\Ga)}^2 + \KAPPA \|\gamma(\ui)\|_{L_2(\Ga)}^2 \\
    &\leq\ \|\nb(\ui)\|_{L_2(\Om)}^2 + \BETA \|\nb_\Ga(\ui)\|_{L_2(\Ga)}^2 + \KAPPA c_{\Om}^2 \|\ui\|_{H^1(\Om)}^2 \\[1mm]
   & \leq\ (1+\KAPPA c_{\Om}^2) Ch^2 \|u\|_{H^2(\Om)}^2 + \BETA C h^2 \| u\|_{H^2(\Ga)}^2 ,
\end{array}
\end{equation}
which yields the stated result.
\end{proof}

We note that the order in $h$ increases from 1 to $m$  if finite elements of degree $m$ are used.

\medskip
Using the above result for the Ritz map we prove a first-order error estimate in the natural norms for the spatial semi-discretisation of \eqref{heat-dynbc}.
\begin{theorem}%[Semidiscrete error bounds]
\label{theorem-spatial-error-est-linear}
   If the solution of the parabolic problem with dynamic boundary conditions \eqref{heat-dynbc} is sufficiently regular, then the error of the semi-discretisation \eqref{heat-dynbc-semi-discrete} satisfies the first-order bound
    \begin{equation}
    \label{eq-semi-discr-error-bound}
        \begin{alignedat}{1}
            &\|u_h\t - u\t\|_{L_2(\Om)}^2 + \MU \| \gamma(u_h\t - u\t)\|_{L_2(\Ga)}^2  \\
            & + \int_0^t \!\! \bigg(\big\|\nb\big(u_h(s) - u(s)\big)\big\|_{L_2(\Om)}^2 + \BETA \big\|\nb_\Ga\big(u_h(s) - u(s)\big)\big\|_{L_2(\Ga)}^2 + \KAPPA\|\gamma(u_h(s) - u(s))\|_{L_2(\Ga)}^2 \bigg) \d s \leq C h^2,
        \end{alignedat}
    \end{equation}
    for $0<t\leq T$, where the constant $C$ is independent of $h$,
    but depends on $T$.
\end{theorem}

\medskip
\begin{proof} Using the error bound of the Ritz projection given in Lemma~\ref{lemma-Ritz-error}, the proof uses the standard argument based on the energy estimate; cf.~\cite{Thomee}. We include the short proof for the convenience of the reader and as a later reference point.
We use the decomposition
\begin{equation*}
    u_h - u = (u_h - R_h u) + (R_h u - u),
\end{equation*}
where the  second term can be estimated using Lemma \ref{lemma-Ritz-error}.
%, as
%\begin{align*}
%    &\frac{1}{2}\|R_h u\t - u\t\|_{L_2(\Om)}^2 + \frac{\MU}{2} \| \gamma(R_h u\t - u\t)\|_{L_2(\Ga)}^2  \\
%    & + h^2 \int_0^t \!\! \bigg( \big\|\nb\big(R_h u(s) - u(s)\big)\big\|_{L_2(\Om)}^2 +\BETA \big\|\nb_\Ga\big(R_h u(s) - u(s)\big)\big\|_{L_2(\Ga)}^2 + \KAPPA \|\gamma(R_h u(s) - u(s))\|_{L_2(\Ga)}^2 \bigg) \d s \leq c h^4.
%\end{align*}
%
Let us denote the first term in the error as $e_h\t:=u_h\t-R_h u\t$. Using the definition of the Ritz projection gives us the error equation
\begin{align*}
    (\dot e_h\t,v_h) + a(e_h\t,v_h) =&\ (\dot u\t - R_h \dot u\t,v_h) \qquad\forall\, v_h\in V_h,
\end{align*}
and the energy estimate \eqref{en-est-h} (with $\alpha=1$ and $c=0$) yields
\begin{equation}\label{eh-est}
|e_h(t)|^2 + \int_0^t \| e_h(s) \|^2 \,\d s \le |e_h(0)|^2 + \int_0^t \|\dot u(s) - R_h \dot u(s)\|_*^2 \,\d s.
\end{equation}
Since $u_h(0)=P_h u(0)$ is the $H$-orthogonal projection of $u(0)$, we have
$$
|e_h(0)| \le |P_h u(0) - u(0)| + |u(0)-R_h u(0)| \le 2 |u(0)-R_h u(0)|.
$$
Moreover, both the dual norm $\|\cdot\|_*$ and the $H$-norm $|\cdot|$ are weaker than the $V$-norm $\|\cdot\|$:
$C^{-1}\|v\|_* \le  |v|\le C \|v\|$ for all $v\in V$.
By Lemma \ref{lemma-Ritz-error}, both  terms on the right-hand side of the energy estimate are thus $O(h^2)$ if the solution $u$ is sufficiently regular.
\end{proof}

\begin{remark}
    The Ritz projection can be defined for the general case $\KAPPA\in\R$. By introducing the positive definite form $a^*(\cdot,\cdot) := a(\cdot,\cdot) + c\; (\cdot,\cdot)$ ($c$ is from the G\aa rding inequality), and defining the Ritz projection with respect to this modified form,
    \begin{equation*}
        a^*(R_hu,v_h) = a^*(u,v_h) \qquad \forall v_h\in V_h,
    \end{equation*}
error estimates for the Ritz map and the semi-discrete error bounds can be shown as above.
\end{remark}

\subsubsection{Second-order error bound in the case of Wentzell boundary conditions}\label{subsubsec:second-wentzell}

We consider the spatial semi-discretisation of problem \eqref{heat-dynbc} with $\BETA=0$ on a polygonal domain using linear finite elements. We assume again $\kappa>0$ for ease of presentation. We will show second-order error bounds in the $L_2(\Omega)\oplus H^{-1/2}(\Gamma)$ norm, using an unusual variant of the Aubin-Nitsche duality argument. We need the following $H^2$-regularity condition.

\begin{Condition}
\label{condition-regularity-0}
    There exists a constant $C_2<\infty$ such that for every  $g_\Om\in L_2(\Om)$ and $g_\Ga\in H^{1/2}(\Gamma)$, the weak solution $\phi$ of the Poisson equation with Robin boundary condition
    \begin{equation}
    \label{eq-elliptic-problem}
        \begin{alignedat}{3}
            -\laplace \phi =&\ g_\Om \qquad &\ &\ \textrm{in } \Om \\
            \pa_\nu \phi  + \KAPPA \phi = &\ g_\Ga \qquad &\ &\ \textrm{on } \Ga
        \end{alignedat}
    \end{equation}
   is in $H^2(\Omega)$ and is bounded by
   \begin{equation}
    \label{eq-regurality-estimate}
        \|\phi\|_{H^2(\Om)}^2  \leq C_2 \Big( \|g_\Om\|_{L_2(\Om)}^2 + \|g_\Ga\|_{H^{1/2}(\Ga)}^2 \Big).
    \end{equation}
 \end{Condition}
\begin{remark} Condition \ref{condition-regularity-0} is known to be satisfied in the case of a smooth domain $\Omega$; see \cite{Taylor}.  We would expect that it also holds for convex polygonal domains, but we are not aware of a reference for such a result.
\end{remark}

%The case $\beta\equiv0$ can be found in \cite{Taylor}.

%\begin{definition}[Ritz projection]
%\label{def-Ritz}
%    For any given $u\in V$ there exists a unique $R_h u\in V_h$ \st\ for all $v_h\in V_h$
%    \begin{equation}
%    \label{ritz}
%        \begin{aligned}
%            &\ \int_\Om \nb (R_h u) \cdot \nb v_h \d x + \BETA \int_\Ga \nb_\Ga(R_h u) \cdot \nb_\Ga v_h \d \sigma + \KAPPA\int_\Ga (\gamma (R_h u))(\gamma v_h) \d \sigma \\
%            =&\ \int_\Om \nb u \cdot \nb v_h \d x + \BETA \int_\Ga \nb_\Ga(u) \cdot \nb_\Ga v_h \d \sigma + \KAPPA\int_\Ga (\gamma u)(\gamma v_h) \d \sigma
%        \end{aligned}
%    \end{equation}
%    holds, i.e.
%    \begin{equation*}
%        a(R_h u,v_h)=a(u,v_h).
%    \end{equation*}
%\end{definition}

\bigskip
The following estimate then holds for the error of the Ritz projection.
\begin{lemma}
\label{lemma-Ritz-error-0}
 If Condition \ref{condition-regularity-0} is satisfied, then the error of the Ritz projection \eqref{ritz} corresponding to the bilinear form \eqref{a-heat-dynbc}  with $\BETA=0$ satisfies the  second-order  bound
   $$     \|\ur\|_{L_2(\Om)}^2 + \MU \|\gamma(\ur)\|_{H^{-1/2}(\Ga)}^2
%        \qquad \qquad \quad  &\ \\
%        + h^2\Big(\|\nb(\ur)\|_{L_2(\Om)}^2 + \BETA \|\nb_\Ga(\ur)\|_{L_2(\Ga)}^2 + &\ \KAPPA \|\gamma(\ur)\|_{L_2(\Ga)}^2\Big) \\
         \leq C h^4 \, \|u\|_{H^2(\Om)}^2  ,
%         \Big( \|u\|_{H^2(\Om)}^2 + \|\nb_\Ga^2 u\|_{L_2(\Ga)}^2 + h^2 \|\nb_\Ga u\|_{L_2(\Ga)}^2 \Big),
   $$% \end{align*}
    where the constant $C$ is independent of $h$  and $u\in H^2(\Omega)$.
\end{lemma}

\begin{proof}
 In the spirit of the Aubin--Nitsche duality argument we consider the elliptic problem, for $g=u-R_hu$,
 $$
   \begin{alignedat}{3}
            -\laplace \phi =&\ g \qquad &\ &\ \textrm{in } \Om \\
            \pa_\nu \phi  + \KAPPA \phi = &\ \mu(-\Delta_\Gamma + I)^{-1/2}\gamma g \qquad &\ &\ \textrm{on } \Ga,
        \end{alignedat}
$$
which has the weak formulation
$$
a(\phi,v) = (g,v)_{L_2(\Omega)} + \mu (\gamma g, \gamma v)_{H^{-1/2}(\Gamma)}
\qquad \forall v \in H^1(\Omega),
$$
where
$$
(\gamma g, \gamma v)_{H^{-1/2}(\Gamma)}  =
\bigl( (-\Delta_\Gamma + I)^{-1/2}\gamma g , \gamma v \bigr)_{L_2(\Gamma)}.
$$
Choosing $v=g=u-R_hu$ we obtain
$$
  \|g\|_{L_2(\Om)}^2 + \MU \|\gamma g\|_{H^{-1/2}(\Ga)}^2 = a(\phi,g) = a(\phi-R_h\phi,g),
$$
using the Galerkin orthogonality \eqref{ritz} in the last equality. This expression is further estimated, using subsequently the Cauchy--Schwarz inequality,
Lemma~\ref{lemma-Ritz-error}, Condition~\ref{condition-regularity-0}, and the relation between the $H^{1/2}(\Gamma)$ and $H^{-1/2}(\Gamma)$ norms:
\begin{align*}
a(\phi-R_h\phi,g) & = a(\phi-R_h\phi,u-R_hu)
\\
&\le \|\phi-R_h\phi\| \cdot \| u-R_hu \| \\[1mm]
&\le Ch \| \phi\|_{H^2(\Omega)} \cdot Ch \|u \|_{H^2(\Omega)} \\
&\le C' h^2 \Bigl(   \|g\|_{L_2(\Om)}^2 + \MU \|(-\Delta_\Gamma + I)^{-1/2}\gamma g\|_{H^{1/2}(\Ga)}^2  \Bigr)^{1/2}  \|u \|_{H^2(\Omega)} \\
&= C' h^2 \Bigl(   \|g\|_{L_2(\Om)}^2 + \MU \|\gamma g\|_{H^{-1/2}(\Ga)}^2  \Bigr)^{1/2}  \|u \|_{H^2(\Omega)},
\end{align*}
and dividing through yields the stated result.
\end{proof}

Lemma~\ref{lemma-Ritz-error-0} yields the following error bound for the spatial semi-discretisation of \eqref{heat-dynbc} with $\BETA=0$.
\begin{theorem}
\label{theorem-spatial-error-est-linear-0}
   If the solution of the parabolic problem with dynamic boundary conditions \eqref{heat-dynbc} with $\BETA=0$ is sufficiently regular and if Condition \ref{condition-regularity-0} is satisfied, then the error of the semi-discretisation \eqref{heat-dynbc-semi-discrete} with starting value $u_h(0)=R_h u(0)$ satisfies the second-order error bound
    \begin{equation}
    \label{eq-semi-discr-error-bound-linear-0}
        \begin{alignedat}{1}
            &\|u_h\t - u\t\|_{L_2(\Om)}^2 + \MU \| \gamma(u_h\t - u\t)\|_{H^{-1/2}(\Ga)}^2   \leq C h^4
        \end{alignedat}
    \end{equation}
    for $0\le t\leq T$, where the constant $C$ is independent of $h$.
    \end{theorem}

\medskip
\begin{proof} We return to the proof of Theorem~\ref{theorem-spatial-error-est-linear} and bound $R_hu(t)- u(t)$ using Lemma~\ref{lemma-Ritz-error-0}. In the energy bound of the error \eqref{eh-est} we note that $e_h(0)=u_h(0)-R_hu(0)=0$ by assumption, and
$$
\|e_h\t \|_{L_2(\Om)}^2 + \MU \| \gamma e_h\t \|_{H^{-1/2}(\Ga)}^2 \le
\|e_h\t \|_{L_2(\Om)}^2 + \MU \| \gamma e_h\t \|_{L_2(\Ga)}^2 = |e_h\t|^2.
$$
Since
\begin{align*}
(w,v) &= (w,v)_{L_2(\Om)} + \mu (\gamma w, \gamma v)_{L_2(\Ga)}
\\
&\le \| w \|_{L_2(\Om)} \, \| w \|_{L_2(\Om)} +
\mu \|\gamma w\|_{H^{-1/2}(\Ga)}\,  \|\gamma v\|_{H^{1/2}(\Ga)}
\\
&\le C \Bigl( \| w \|_{L_2(\Om)} ^2 + \mu \|\gamma w\|_{H^{-1/2}(\Ga)}^2 \Bigr)^{1/2}
\| v \|_{H^1(\Om)},
\end{align*}
we obtain
$$
\| w \|_* \le C \Bigl( \| w \|_{L_2(\Om)} ^2 + \mu \|\gamma w\|_{H^{-1/2}(\Ga)}^2 \Bigr)^{1/2},
$$
which for $w=\dot u(s) - R_h\dot u(s)$, for $0\le s \le t \le T$, is bounded by $O(h^2)$ by Lemma~\ref{lemma-Ritz-error-0}. Hence \eqref{eh-est} together with the above estimates yields the result.
\end{proof}

\subsubsection{Second-order error bound in the case of bulk-surface diffusion coupling} \label{subsubsec:second-bsdc}
In the case of $\BETA>0$ in \eqref{heat-dynbc} we use the following $H^2$-regularity condition.
   \begin{Condition}\label{condition-regularity-beta}
       There exists a constant $C_2<\infty$ such that for every  $g_\Om\in L_2(\Om)$ and $g_\Ga\in H^{1/2}(\Gamma)$, the weak solution $\phi$ of the elliptic problem, with $\BETA>0$ and $\KAPPA>0$,
    \begin{equation}
    \label{eq-elliptic-problem-beta}
        \begin{alignedat}{3}
            -\laplace \phi =&\ g_\Om \qquad &\ &\ \textrm{in } \Om \\
            \pa_\nu \phi  - \BETA \Delta_\Ga\phi + \KAPPA \phi = &\ g_\Ga \qquad &\ &\ \textrm{on } \Ga,
        \end{alignedat}
    \end{equation}
   is in $H^2(\Omega)$, has trace in $H^2(\Gamma)$ and is
   bounded by
    \begin{equation}
    \label{eq-regurality-estimate-beta}
        \|\phi\|_{H^2(\Om)}^2 +  \|\gamma \phi\|_{H^2(\Ga)}^2 \leq C_2 \Big( \|g_\Om\|_{L_2(\Om)}^2 + \|g_\Ga\|_{L_2(\Ga)}^2 \Big).
    \end{equation}
\end{Condition}
\begin{remark}
It can be shown that this condition is satisfied for smooth domains $\Om$, adapting the standard $H^2$-regularity proof for the heat equation with Dirichlet or Neumann boundary conditions as given, e.g., in \cite[Section 9.6]{Brezis}. We are not aware of a suitable reference for (convex) polygonal domains.
\end{remark}

\begin{lemma}
\label{lemma-Ritz-error-beta}
 If Condition \ref{condition-regularity-beta} is satisfied, then the error of the Ritz projection \eqref{ritz} corresponding to the bilinear form \eqref{a-heat-dynbc}  with $\BETA>0$ satisfies the  second-order  bound
   $$     \|\ur\|_{L_2(\Om)}^2 + \MU \|\gamma(\ur)\|_{L_2(\Ga)}^2
%        \qquad \qquad \quad  &\ \\
%        + h^2\Big(\|\nb(\ur)\|_{L_2(\Om)}^2 + \BETA \|\nb_\Ga(\ur)\|_{L_2(\Ga)}^2 + &\ \KAPPA \|\gamma(\ur)\|_{L_2(\Ga)}^2\Big) \\
         \leq C h^4 \, \bigl( \|u\|_{H^2(\Om)}^2  + \|u\|_{H^2(\Ga)}^2),
%         \Big( \|u\|_{H^2(\Om)}^2 + \|\nb_\Ga^2 u\|_{L_2(\Ga)}^2 + h^2 \|\nb_\Ga u\|_{L_2(\Ga)}^2 \Big),
   $$% \end{align*}
    where the constant $C$ is independent of $h$  and $u\in H^2(\Omega)$ with $\gamma u \in H^2(\Ga)$.
\end{lemma}

\begin{proof} The proof uses the standard Aubin--Nitsche duality argument with $g_\Om=\ur$ and $g_\Ga=\gamma(\ur)$ in
 \eqref{eq-elliptic-problem-beta}; cf.~the proof of Lemma~\ref{lemma-Ritz-error-0}. We omit the details.
\end{proof}

Together with the proof of Theorem~\ref{theorem-spatial-error-est-linear},   Lemma~\ref{lemma-Ritz-error-beta} yields the following error bound for the spatial semi-discretisation of \eqref{heat-dynbc} with $\BETA>0$.
\begin{theorem}
\label{theorem-spatial-error-est-linear-beta}
   If the solution of the parabolic problem with dynamic boundary conditions \eqref{heat-dynbc} with $\BETA>0$ is sufficiently regular and if Condition \ref{condition-regularity-beta} is satisfied, then the error of the semi-discretisation \eqref{heat-dynbc-semi-discrete} satisfies the second-order error bound
    \begin{equation}
    \label{eq-semi-discr-error-bound-linear-beta}
        \begin{alignedat}{1}
            &\|u_h\t - u\t\|_{L_2(\Om)}^2 + \MU \| \gamma(u_h\t - u\t)\|_{L_2(\Ga)}^2   \leq C h^4
        \end{alignedat}
    \end{equation}
    for $0\le t\leq T$, where the constant $C$ is independent of $h$.
    \end{theorem}

\subsubsection{Mass lumping} \label{subsubsec:mass-lumping}

In many situations it is convenient to replace the mass matrix $\bfM$ of \eqref{eq-mass-stiffness-matrix-def} by a diagonal matrix. This can be achieved by replacing the $L_2(\Omega)\oplus L_2(\Gamma)$ inner product $(\cdot,\cdot)$ of \eqref{b-heat-dynbc} by  a suitable quadrature.
%The use of exponential integrators would not be efficient with a non diagonal mass matrix $\bfM$, since its inverse would be involved in the computations. A natural way to circumvent this is to approximate $\bfM$ with a diagonal matrix. This can be achieved and analysed within

We follow
the framework of  lumped-mass methods as presented in \cite[Section 15]{Thomee}.
We use the analogue  of the trapezoidal rule in $d$ and $d-1$ dimensions:
\begin{equation*}
    Q_E(f) = {\rm vol}_d(E) \,\frac{1}{d+1} \sum_{j=1}^{d+1} f(x_j) \andquad  Q_e(f) = {\rm vol}_{d-1}(e)\, \frac1d\sum_{j=1}^{d}  f(x_j),
\end{equation*}
where ${\rm vol}_d$ denotes the $d$-dimensional volume, $E$ is an element of the quasi-uniform triangulation $\calT_h$ with vertices $x_j$, and $e\in \pa\calT_h$ is a boundary element. Then we use the above quadratures to define the lumped mass approximation of the inner product $(\cdot,\cdot)$ on $H=L_2(\Omega)\oplus L_2(\Gamma)$:
\begin{equation}
\label{eq-lumped-mass-form}
  (u,v)_\lm := \sum_{E\in\calT_h} Q_E(uv) + \MU \sum_{e\in \pa\calT_h} Q_e(uv),
\end{equation}
We denote the induced (lumped mass) norm by $|\,\cdot\,|_\lm$.

\bigskip
The lumped-mass semi-discrete problem determines $u_h:[0,T] \rightarrow V_h$ \st
\begin{equation}
\label{eq-heat-dynbc-lumped-mass}
    \begin{alignedat}{3}
        (\dot u_h\t,v_h)_\lm + a(u_h\t,v_h) =&\ (f\t,v_h)_\lm \qquad&\ &\ \forall v_h\in V_h \qquad (0<t\leq T)\\
        (u_h(0),v_h)_\lm=&\ (u_0,v_h)_\lm \qquad&\ &\ \forall v_h\in V_h
    \end{alignedat}
\end{equation}
Then the nodal vector $\bfu\t$ satisfies the corresponding linear ODE system
\begin{equation}
\label{eq-lumped-mass-matrix-form}
    \M\dot\bfu(t) + \bfA \bfu(t) = \bfb(t),
\end{equation}
where the new mass matrix, which we denote again by $\bfM$, has the entries
\begin{equation*}
    m_{ij}=(\vphi_j,\vphi_i)_\lm\,, \andquad b_i (t)= (f\t,\vphi_i)_\lm \qquad (i,j=1,\dotsc,N).
\end{equation*}
The new mass matrix is diagonal, since the product $\vphi_j\vphi_i$ vanishes at all nodes with $j\neq i$. The stiffness matrix $\bfA$ is defined as before, see \eqref{eq-mass-stiffness-matrix-def}. The initial values are now chosen such that $u_h(x_j,0)=u(x_j,0)$ in each node of the triangulation.

Before proving error estimates for the semi-discretisation with mass lumping we formulate an important technical lemma. Here $\| \cdot \|$ is again the norm induced by the bilinear form $a(\cdot,\cdot)$ of \eqref{a-heat-dynbc}.

\begin{lemma}
\label{lemma-LM-quadrature-error}
    If $\BETA>0$ in \eqref{b-heat-dynbc}, then the quadrature error  $\calE(v_h,w_h)=(v_h,w_h)_\lm-(v_h,w_h)$ is bounded by    \begin{equation*}
        |\calE(v_h,w_h)| \leq C h^2 \|v_h\| \ \|w_h\| \for v_h,w_h \in V_h.
    \end{equation*}
\end{lemma}

\begin{proof} We separate the mass lumping errors in the bulk and on the surface as $\calE=\calE_\Om + \mu \calE_\Ga$.
The proof of Lemma 15.1 in \cite{Thomee} shows that
\begin{equation}\label{quad-err-vw}
\begin{aligned}
&|\calE_\Om(v_h,w_h)| \leq C h^2 \|v_h\|_{H^1(\Om)} \ \|w_h\|_{H^1(\Om)}
\\[1mm]
&|\calE_\Ga(v_h,w_h)| \leq C h^2 \|v_h\|_{H^1(\Ga)} \ \|w_h\|_{H^1(\Ga)}
\end{aligned}
 \quad \for v_h,w_h \in V_h.
\end{equation}
This yields the stated bound for $\calE$ for the norm $\| \cdot \|$  induced by the bilinear form $a(\cdot,\cdot)$ of \eqref{a-heat-dynbc}, provided that $\BETA>0$.
\end{proof}

\begin{theorem}
\label{theorem-lumped-mass-error-est}
If the solution of the parabolic problem with dynamic boundary conditions \eqref{heat-dynbc}  is sufficiently regular, then the error of the semi-discretisation  with mass lumping \eqref{eq-heat-dynbc-lumped-mass} satisfies a first-order error bound as in
Theorem~\ref{theorem-spatial-error-est-linear}.
If $\BETA>0$ and Condition~\ref{condition-regularity-beta} is satisfied, then there is  also a second-order error bound as in Theorem~\ref{theorem-spatial-error-est-linear-beta}.
%
%    \begin{equation}
%    \label{eq-lumped-mass-error-bound}
%        \begin{alignedat}{1}
%            &\|u_h\t - u\t\|_{L_2(\Om)}^2 + \frac{\MU}{2} \| \gamma(u_h\t - u\t)\|_{L_2(\Ga)}^2  \\
%            & + h^2 \int_0^t \!\! \bigg(\big\|\nb\big(u_h(s) - u(s)\big)\big\|_{L_2(\Om)}^2 + \BETA \big\|\nb_\Ga\big(u_h(s) - u(s)\big)\big\|_{L_2(\Ga)}^2 + \KAPPA\|\gamma(u_h(s) - u(s))\|_{L_2(\Ga)}^2 \bigg) \d s \leq C h^4,
%        \end{alignedat}
%    \end{equation}
%    or formulated with the abstract norms
%    \begin{equation*}
%        |u_h\t - u\t| + h^2 \int_0^t \|u_h(s) - u(s)\| \d s \leq C h^4,
%    \end{equation*}
%    for $0<t\leq T$, where the constant $C$ is independent of $h$, but depends on $T$.
\end{theorem}

\begin{proof}
The proof differs from that of Theorem \ref{theorem-spatial-error-est-linear} in that
the error equation for the lumped mass case for $e_h\t=u_h\t-R_h u\t$ contains extra quadrature error terms: for all $v_h\in V_h,$
$$
    (\dot e_h\t,v_h)_\lm + a(e_h\t,v_h) =\ (\dot u\t - R_h \dot u\t,v_h) - \calE(R_h \dot u\t,v_h) + \calE(I_hf\t,v_h) + (I_hf(t)-f(t),v_h).
$$
We first consider the case $\BETA>0$. The energy estimate obtained by testing with $v_h=e_h(t)$ and using Lemma~\ref{lemma-LM-quadrature-error} to bound, for $w_h=R_h \dot u(t) - I_hf(t)$,
$$
\calE(w_h,e_h) \le C h^2 \|w_h\|\, \|e_h\| \le \tfrac14 \|e_h\|^2 + C'h^4 \|w_h\|^2
$$
then yield the result with $|\cdot|_\lm$ in place of $|\cdot|$.  Together with  the inverse estimate (valid for  quasi-uniform triangularisations)
$$
\| v_h \| \le Ch^{-1} |v_h| \for v_h\in V_h,
$$
Lemma~\ref{lemma-LM-quadrature-error} further provides the equivalence of the norms $|\cdot|$ and $|\cdot|_\lm$ on $V_h$, uniformly in $h$.

In the case $\BETA=0$ we use \eqref{quad-err-vw} and the inverse estimate to obtain (with different constants that are all denoted~$C$)
\begin{eqnarray*}
\calE(w_h,e_h) &\le& Ch^2 \bigl( \| w_h\|_{H^1(\Om)}  \| e_h\|_{H^1(\Om)} + \| w_h\|_{H^1(\Ga)}  \| e_h\|_{H^1(\Ga)} \bigr)
\\
&\le &Ch^2 \bigl( \| w_h\|_{H^1(\Om)}  \| e_h\|_{H^1(\Om)} + \| w_h\|_{H^1(\Ga)}  \, Ch^{-1} \| e_h\|_{L_2(\Ga)} \bigr)
\\
&\le & Ch  \bigl( \| w_h\|_{H^1(\Om)}  + \| w_h\|_{H^1(\Ga)}  \bigr) \|e_h\|
\\
&\le &  \tfrac14 \|e_h\|^2 + Ch^2 \bigl( \| w_h\|_{H^1(\Om)} ^2 + \| w_h\|_{H^1(\Ga)}  ^2\bigr).
\end{eqnarray*}
This is sufficient to obtain the first-order error bound with $|\cdot|_\lm$ in place of $|\cdot|$. Estimating the quadrature error in the above way still yields the $h$-uniform equivalence of the norms $|\cdot|$ and $|\cdot|_\lm$ on $V_h$.
\end{proof}

\begin{remark}
A second-order error bound is not obtained for mass-lumping in the case of Wentzell boundary conditions ($\BETA=0$). A second-order error bound does hold, however, if mass-lumping is done only in the bulk, but not on the surface. For such a partial mass-lumping
Lemma~\ref{lemma-LM-quadrature-error}  remains valid for $\BETA=0$.
\end{remark}

\subsection{Spatial semi-discretisation of non-autonomous linear problems}
\label{subsection-semi-discr-nonauto}

The results of the previous subsection generalize to the heat equation with non-autonomous dynamic boundary conditions \eqref{heat-dynbc-t} on a polygonal domain. Using  the time-dependent Ritz projection corresponding to the bilinear form $a(t;\cdot,\cdot)$ given by \eqref{a-heat-dynbc-t}, we show optimal-order error bounds for the space discretisation. Since the proofs of these results are very similar to the ones presented above, we only focus on the major differences. Again for the ease of presentation we assume $\KAPPA(x,t)>0$ and bounded by $\overline{\KAPPA}$, while $\BETA$ is bounded by $\overline{\BETA}$, hence the form $a(t;\cdot,\cdot)$ induces the norm $\| v \|_t=a(t;v,v)^{1/2}$ on $V=\{v \in H^1(\Omega)\,:\, \sqrt{\beta} \,\nabla_\Gamma(\gamma v) \in L_2(\Gamma)\}$. The inner product on $H=L_2(\Omega)\oplus L_2(\Gamma)$ is given by \eqref{b-heat-dynbc-t}, and induces the norm $|w|_t=m(t;v,v)^{1/2}$ on $H=L_2(\Om)\oplus L_2(\Ga)$. Note that both norms are time dependent, but in each family of norms the members are equivalent to each other uniformly in $t$.
%In order to avoid the need of boundary approximations we assume $\Om$ to be a polyhedral domain.

We work with the same family of quasi-uniform triangulations as in the previous section, and also use the same piecewise linear finite element basis functions.
The  semi-discretisation of the non-autonomous problem is
\begin{equation}
\label{heat-dynbc-t-semi-discrete}
    \begin{alignedat}{1}
        m(t;\dot u_h(t),v_h) + a(t;u_h(t),v_h) &= m(t;f\t,v_h) \qquad \forall \,v_h\in V_h \qquad (0<t\le T)\\
        m(0;u_h(0),v_h) &= m(0;u_{0} ,v_h) \qquad\ \, \, \forall \,v_h\in V_h .
%u_h(0)=u_{0,h} \in V_h.
    \end{alignedat}
\end{equation}
This now yields a non-autonomous system of linear ordinary differential equations for the coefficient vector $\bfu(t)$,
\begin{equation*}
%\label{eq-matrix-form}
    \bfM\t\dot \bfu\t + \bfA\t \bfu\t = \bfb\t,
\end{equation*}
where $\bfM\t$ and $\bfA\t$ are the time-dependent mass and stiffness matrix, \resp. Their entries are given as
\begin{equation}
\label{eq-mass-stiffness-matrix-def-t}
    \begin{gathered}
        m_{ij}\t = m(t;\vphi_j,\vphi_i) \andquad
        a_{ij}\t = a(t;\vphi_j,\vphi_i)
    \end{gathered}
    \qquad (i,j=1,2,\dotsc,N),
\end{equation}
and $\bfb\t$ is the load vector, with entries
\begin{equation*}
    b_i\t = m(t;f\t,\vphi_i) \qquad (i=1,2,\dotsc,N).
\end{equation*}

In order to prove optimal order error estimates we begin by studying the Ritz projection $R_h(t):V\to V_h$ with respect to the elliptic bilinear form $a(t;\cdot,\cdot)$ of \eqref{a-heat-dynbc-t}, for $0\leq t\leq T$. For every $u\in V$, the projected function $R_h(t) u\in V_h$ is defined as the unique finite element function in $V_h$ such that
%\begin{definition}[non-autonomous Ritz projection]
    \begin{equation}\label{ritz-t}
        a(t;R_h(t) u,v_h)=a(t;u,v_h) \qquad \forall v_h\in V_h.
    \end{equation}
 The error bound of Lemma~\ref{lemma-Ritz-error} extends to the time-dependent case using the same proof and observing the uniformity of the bounds in $t$.

\begin{lemma} \label{lemma-Ritz-t}
   The error of the Ritz projection corresponding to the bilinear form  \eqref{a-heat-dynbc-t} satisfies a first-order error bound in the time-dependent energy norm for $0\le t\leq T$,
$$
        \| u-R_h(t)u \|_t^2 \leq \ C h^2 \Bigl( \|u\|_{H^2(\Om)}^2 + \bar\BETA\|\gamma u\|_{H^2(\Ga)}^2 \Bigr),
 $$
 where  $C$ is independent of $h$ and $t$. Here, $\bar\BETA$ is an upper bound of $\BETA(\cdot,t)$.
    \end{lemma}

We  have $\frac d{dt}(R_h(t)u(t)) \ne R_h(t)\dot u(t)$ in general, but we note the following bound.

\begin{lemma}\label{lemma-Ritz-dot}
  The error of the time derivative of the Ritz projection corresponding to the bilinear form  \eqref{a-heat-dynbc-t} satisfies a first-order error bound in the time-dependent energy norm for continuously differentiable $u:[0,T]\to H^2(\Om)$ and $0\le t\leq T$,
$$
        \| \dot u(t)-\tfrac d{dt}\bigl(R_h(t)u(t)\bigr) \|_t^2 \leq \ C h^2 \Bigl( \|u(t)\|_{H^2(\Om)}^2 + \bar\BETA\|\gamma u(t)\|_{H^2(\Ga)}^2 + \|\dot u(t)\|_{H^2(\Om)}^2 + \bar\BETA\|\gamma \dot u(t)\|_{H^2(\Ga)}^2\Bigr),
 $$
 where  $C$ is independent of $h$ and $t$. Here, $\bar\BETA$ is again an upper bound of $\BETA(\cdot,t)$.
\end{lemma}

\begin{proof} Differentiating \eqref{ritz-t} with respect to $t$ yields
$$
a\bigl(t; \dot u(t) - \tfrac d{dt}(R_h(t)u(t)),v_h\bigr) = - {\partial_t}a\bigl(t;u(t)-R_h(t)u(t),v_h\bigr) \qquad \forall v_h\in V_h.
$$
We then have, omitting the argument $t$,
\begin{eqnarray*}
 \| \dot u - \tfrac d{dt}(R_hu)\|_t^2 &=& a(t;  \dot u - \tfrac d{dt}(R_hu), \dot u - R_h\dot u) +
 a(t;  \dot u - \tfrac d{dt}(R_hu),R_h\dot u - \tfrac d{dt}(R_hu))
 \\
 &=& a(t;  \dot u - \tfrac d{dt}(R_hu), \dot u - R_h\dot u) - {\partial_t}a\bigl(t;u-R_hu,R_h\dot u - \tfrac d{dt}(R_hu))
 \\
 &=& a(t;  \dot u - \tfrac d{dt}(R_hu), \dot u - R_h\dot u) - {\partial_t}a\bigl(t;u-R_hu,R_h\dot u - \dot u) -
 {\partial_t}a\bigl(t;u-R_hu,\dot u - \tfrac d{dt}(R_hu)).
\end{eqnarray*}
With \eqref{a-dot-bound} and Lemma~\ref{lemma-Ritz-t}, this is estimated as
\begin{eqnarray*}
\| \dot u - \tfrac d{dt}(R_hu)\|_t^2 &\le& \| \dot u - \tfrac d{dt}(R_hu)\|_t \,
\sqrt C h \Bigl( \|\dot u\|_{H^2(\Om)}^2 + \bar\BETA\|\gamma \dot u\|_{H^2(\Ga)}^2\Bigr)^{1/2}
\\
&& +\
M_1' C h^2 \Bigl( \|u\|_{H^2(\Om)}^2 + \bar\BETA\|\gamma u\|_{H^2(\Ga)}^2 \Bigr)^{1/2}
\Bigl( \|\dot u\|_{H^2(\Om)}^2 + \bar\BETA\|\gamma \dot u\|_{H^2(\Ga)}^2\Bigr)^{1/2}
\\
&& +\
M_1' \sqrt C h \Bigl( \|u\|_{H^2(\Om)}^2 + \bar\BETA\|\gamma u\|_{H^2(\Ga)}^2 \Bigr)^{1/2}\,
\| \dot u - \tfrac d{dt}(R_hu)\|_t,
\end{eqnarray*}
which yields the result.
\end{proof}

We then have the following non-autonomous version of Theorem~\ref{theorem-spatial-error-est-linear}.
\begin{theorem}
\label{theorem-spatial-error-est-nonauto}
    If the solution of the linear non-autonomous parabolic equation with dynamic boundary condition  \eqref{heat-dynbc-t} is sufficiently regular, then the error of the semi-discretisation \eqref{heat-dynbc-t-semi-discrete} satisfies the first-order error bound
    \begin{equation*}
        \begin{alignedat}{1}
            & \|u_h\t - u\t\|_{L_2(\Om)}^2 +  \int_\Ga \MU(.,t)\Big(\gamma u_h(t) - \gamma u(t)\Big)^2\d \sigma + \int_0^t \big\|\nb\big(u_h(s) - u(s)\big)\big\|_{L_2(\Om)}^2 \d s \\
            & + \int_0^t \int_\Ga \BETA(.,s)\bigl(\nb_\Ga u_h(s) - \nb_\Ga u(s)\bigr)^2 \d \sigma \d s + \int_0^t \int_\Ga \KAPPA(.,s)\bigl(\gamma u_h(s) - \gamma u(s)\bigr)^2 \d \sigma \d s \leq C h^2,
        \end{alignedat}
    \end{equation*}
for $0<t\leq T$, where the constant $C$ is independent of $h$ and $t$, but depends on $T$.% and $\overline{\KAPPA}$ and $\overline{\BETA}$.
\end{theorem}

\begin{proof}
%Again we show the abstract estimate using the same error decomposition.
The error equation for $e_h=u_h-R_hu$ is now
\begin{equation*}
    m(t;\dot e_h\t,v_h) + a(t;e_h\t,v_h) = m(t;\dot u\t - \tfrac d{dt}(R_h(t)u(t)),v_h) \qquad \forall v_h\in V_h.
\end{equation*}
The energy estimate \eqref{en-est-t} (with $\alpha=1$ and $c=0$) thus gives us
$$
|e_h(t)|_t^2 +  \int_0^t \| e_h(s) \|_s^2 \, \d s \le e^{2 M_0' t}
\Bigl( |e_h(0)|_0^2 + \int_0^t  \| \dot u(s) - \tfrac d{ds}(R_h(s)u(s)) \|_{*,s}^2\, \d s \Bigr),
$$
and the result follows with Lemmas~\ref{lemma-Ritz-t} and~\ref{lemma-Ritz-dot}.
\end{proof}

Under appropriate $H^2$-regularity conditions, second-order error bounds are obtained in the same way as in Subsections~\ref{subsubsec:second-wentzell}
and~\ref{subsubsec:second-bsdc}.

Mass lumping can be discussed as in Subsection~\ref{subsubsec:mass-lumping}. In particular, the second-order error bound of Theorem~\ref{theorem-spatial-error-est-nonauto} remains valid for mass lumping both in the bulk and on the surface for strictly positive $\beta(x,t)$, and for mass lumping in the bulk, but not on the surface for the case $\beta\equiv 0$.

\subsection{Spatial semi-discretisation of semi-linear  problems}

The above techniques readily extend  to semi-linear   parabolic problems with dynamic boundary conditions.
Let us consider the following semi-linear  problem, which includes the first two examples in Section~\ref{subsec:nonlinear-examples}:
\begin{equation}
\label{eq-general-dynbc-semi-linear -problem}
    \begin{alignedat}{3}
        \pa_t u =&\ \laplace u + f_\Om(u) \qquad &\ &\ \textnormal{in } \ \Om \\
        \MU\, \pa_t u =&\ \BETA \laplace_\Ga u - \KAPPA u +\mu  f_\Ga(u) -\pa_\nu u \qquad &\ &\ \textnormal{in } \ \Ga,
    \end{alignedat}
\end{equation}
with $\MU>0$ and $\BETA\ge 0$, and where the nonlinearities $f_\Omega:\R\to\R$ and $f_\Gamma:\R\to\R$
are continuously differentiable functions, and hence locally Lipschitz continuous.
%
% \begin{align*}
%     (v-w)\cdot(f_\Om(v)-f_\Om(w)) \leq&\ L_\Om(r) |v-w|, \\
%     (v-w)\cdot(f_\Ga(v)-f_\Ga(w)) \leq&\ L_\Ga(r) |v-w|,
% \end{align*}
% for all $v,w\in \R$ \st\ $|v|,|w| \leq r$.
For $u\in C(\bar\Omega)$ we use the notation $f(u)=\big(f_\Om(u), f_\Ga(u)\big)^T\in C(\bar\Omega) \oplus C(\Gamma)$ with $f_\Omega(u)(x)=f_\Omega(u(x))$ for $x\in \bar\Omega$ and $f_\Ga(u)(x)=f_\Ga(u(x))$ for $x\in\Gamma$.
% \begin{equation*}
%     f:=\big(f_\Om,\MU\inv f_\Ga\big)^T.
% \end{equation*}

The bilinear forms corresponding to the problem are the same as in \eqref{a-heat-dynbc}-\eqref{b-heat-dynbc}, again taken with a positive $\KAPPA$ to simplify the presentation (otherwise the linear term can be absorbed in $f_\Gamma(u)$). The finite element semi-discretisation of the semi-linear  problem is
\begin{equation}
\label{semi-linear -dynbc-semi-discrete}
  (\dot u_h(t),v_h) + a(u_h(t),v_h) = (f(u_h\t),v_h) \qquad \forall \,v_h\in V_h \qquad (0<t\le T).
%    \begin{alignedat}{1}
%        (\dot u_h(t),v_h) + a(u_h(t),v_h) &= (f(u_h\t),v_h) \qquad \forall \,v_h\in V_h \qquad (0<t\le T)\\
%        (u_h(0),v_h) &= (u_{0} ,v_h) \qquad\qquad\ \  \forall \,v_h\in V_h .
%%u_h(0)=u_{0,h} \in V_h.
%    \end{alignedat}
\end{equation}
For the coefficient vector $\bfu(t)$ this yields the nonlinear system of ordinary differential equations
\begin{equation*}
%\label{eq-matrix-form-semi-linear }
    \bfM\dot \bfu(t) + \bfA \bfu(t) = \bff(\bfu(t)),
\end{equation*}
where $\bfM$ is the mass matrix, $\bfA$ is the stiffness matrix, both defined in \eqref{eq-mass-stiffness-matrix-def}, and $\bff(\bfu(t))$ is the vector with entries
$f_i(\bfu(t))= (f(u_h\t),\vphi_i)$  ($i=1,2,\dotsc,N$).
We give the semi-linear  analogue  of Theorems \ref{theorem-spatial-error-est-linear}--\ref{theorem-spatial-error-est-linear-beta}.
\begin{theorem}
\label{theorem-spatial-error-est-semi-linear } Consider the semi-linear  equation with dynamic boundary condition \eqref{eq-general-dynbc-semi-linear -problem} over a polygonal domain $\Omega\subset\R^d$ with $d\le 3$. 
The following error bounds hold  if the solution of  \eqref{eq-general-dynbc-semi-linear -problem} is sufficiently regular, if Conditions  \ref{condition-regularity-0} and \ref{condition-regularity-beta} are satisfied and if the initial data satisfy
$
\|u_h(0) - R_hu(0)\|_{L_2(\Om)}^2 + \MU \| \gamma(u_h(0) - R_hu(0))\|_{L_2(\Ga)}^2
\le C_0 h^4$.
Then, the error of the semi-discretisation \eqref{semi-linear -dynbc-semi-discrete} by linear finite elements on quasi-uniform triangulations  satisfies the first-order error bound
%    Let the initial value for the spatially discrete problem be chosen as $u_h(0)=R_hu(0)$, where $R_h$ denotes the Ritz projection defined above. Furthermore if we assume that the solution of the variational problem corresponding to the general equation with dynamic boundary condition, i.e.\ \eqref{eq-general-dynbc-semi-linear -problem}, is sufficiently smooth, and that Condition \ref{condition-regularity} is satisfied. We further assume that it is bounded in $L_\infty(\Om)$, $L_\infty(\Ga)$. Then the error of the semi-discretisation \eqref{semi-linear -dynbc-semi-discrete} is bounded as
    \begin{equation}
    \label{eq-semi-discr-error-bound-semi-linear -1}
        \begin{alignedat}{1}
            &\|u_h\t - u\t\|_{L_2(\Om)}^2 + \MU \| \gamma(u_h\t - u\t)\|_{L_2(\Ga)}^2  \\
            & + \int_0^t \!\! \bigg(\big\|\nb\big(u_h(s) - u(s)\big)\big\|_{L_2(\Om)}^2 + \BETA \big\|\nb_\Ga\big(u_h(s) - u(s)\big)\big\|_{L_2(\Ga)}^2 + \KAPPA\|\gamma(u_h(s) - u(s))\|_{L_2(\Ga)}^2 \bigg) \d s \leq C_1 h^2,
        \end{alignedat}
    \end{equation}
    for $0<t\leq T$ and $0<h\le h_0$ with a sufficiently small $h_0$.
    Moreover,  there is the second-order error bound
    \begin{equation} \label{eq-semi-discr-error-bound-semi-linear -2}
        \|u_h\t - u\t\|_{L_2(\Om)}^2 + \MU \| \gamma(u_h\t - u\t)\|_{H^{-\theta}(\Ga)}^2 \leq C_2 h^4,
    \end{equation}
    where $\theta=1/2$ if $\beta=0$, and $\theta=0$ if $\beta>0$ in \eqref{eq-general-dynbc-semi-linear -problem}.
    The constants $C_1$ and $C_2$ are independent of $h$ and~$t$, but depend on $T$.
\end{theorem}

\begin{proof}
%A slight modification of the proof of the semi-discrete error bounds for the linear case (Theorems \ref{theorem-spatial-error-est-linear}--\ref{theorem-spatial-error-est-linear-beta}) suffices here.
The error equation for $e_h(t)=u_h(t)-R_hu(t)$ reads, with the elliptic bilinear form $a$ of \eqref{a-heat-dynbc} and the inner product \eqref{b-heat-dynbc},
\begin{align*}
    (\dot e_h\t,v_h) + a(e_h\t,v_h) =&\ (\dot u\t - R_h \dot u\t,v_h) + \big(f(u_h\t)-f(u\t),v_h\big)
    \qquad \forall v_h\in V_h.
\end{align*}
We use again the energy estimate, which now becomes
$$
|e_h(t)|^2 + \int_0^t \| e_h(s) \|^2 \,\d s \le |e_h(0)|^2 + \int_0^t \|\dot u(s) - R_h \dot u(s)\|_*^2 \,\d s +
\int_0^t \|f(u_h(s))-f(u(s))\|_*^2 \,\d s.
$$
As long as $\|u(t)\|_{L_\infty(\Om)} \le r$ and $\|u_h(t)\|_{L_\infty(\Om)} \le r$, the local Lipschitz continuity of $f_\Omega$ and $f_\Gamma$ yields (omitting the argument $t$)
$$
\| f(u_h) - f(u) \|_* \le L(r) |u_h-u| \le L(r) |e_h| + L(r) |R_h u -u|.
$$
The maximum-norm boundedness of $u_h$ is ensured by an inverse estimate:
% \begin{equation*}
%     \|u_h\|_{L_\infty} \leq \|u_h-u\|_{L_\infty} + \|u\|_{L_\infty} \leq C h\inv \|u_h-u\|_{H^1} + \|u\|_{L_\infty} \leq C  \|u\|_{H^2} + \|u\|_{L_\infty}.
% \end{equation*}
\begin{eqnarray*}
    \|u_h\|_{L_\infty(\Om)} &\leq& \|u_h-I_hu\|_{L_\infty(\Om)} + \|I_hu\|_{L_\infty(\Om)}
        \\
     &\leq& C h^{-d/2} \|u_h-I_hu\|_{L_2(\Om)} + \|u\|_{L_\infty(\Om)} \\
         &\leq& C h^{-d/2} \|u_h-u\|_{L_2(\Om)} +C h^{-d/2} \|u-I_hu\|_{L_2(\Om)} + \|u\|_{L_\infty(\Om)} \\
    & \leq&      C C_2h^{2-d/2} +C h^{2-d/2} \|u\|_{H^2(\Om)} + \|u\|_{L_\infty(\Om)}   
   \\
    & \leq&   1+\|u\|_{L_\infty(\Om)}  
    \end{eqnarray*}
for sufficiently small $h$, where the last but one inequality holds true as long as    \eqref{eq-semi-discr-error-bound-semi-linear -2} is valid.
The error bounds for the Ritz projection given by Lemmas~\ref{lemma-Ritz-error}--\ref{lemma-Ritz-error-beta}
and a Gronwall inequality conclude the proof.
\end{proof}

Mass lumping can  be studied as before, yielding the semi-linear  analogue  of Theorem~\ref{theorem-lumped-mass-error-est}.

\section{The Ritz map for non-polygonal domains}

For smooth domains the polygonal approximation of the bulk and the surface requires extra care in the error analysis of spatial discretisations; cf., e.g., \cite{Dziuk88,DziukElliott_ESFEM,ElliottRanner}. In this section we will discuss spatial semi-discretisations of problems with dynamic boundary conditions on a smooth domain $\Om$  with boundary $\Ga$.

\subsection{Preparation: Lifts and their approximation estimates}
\newcommand{\DIM}{2} % dimension of our domain

%
% We again consider the problem \eqref{heat-dynbc}, given through the bilinear forms \eqref{a-heat-dynbc} and \eqref{b-heat-dynbc}.
%(in this section the latter one will be denoted by $m(\, \cdot\, , \, \cdot \, )$ instead of $(\, \cdot\, , \, \cdot \, )$).

% %\bigskip
% Following \cite{Dziuk88} (and also \cite{ElliottRanner}), we assume that the boundray surface $\Ga$ is smooth, hence there exists a function $d=d(x)$ such that we have the representation
% \begin{align*}
%     d(x) :=
%     \begin{cases}
%         -\inf\{ |x-y| \, : \, y\in\Ga \}, & \textrm{ if } x\in \Om, \\
%         0, & \textrm{ if } x\in \Ga, \\
%         \phantom{-} \inf\{ |x-y| \, : \, y\in\Ga \}, & \textrm{ if } x\notin \overline{\Om}.
%     \end{cases}
% \end{align*}
% Since $d$ is a signed distance function, for almost every $x\in\Ga$ we have $\nu(x)=\nb d(x)$ for the outward normal. Furthermore there exists a narrow band $U$ around $\Ga$ where $d\in C^2(U)$ holds. For arbitrary $x\in U$ there is a unique $p(x)\in\Ga$ \st
% \begin{equation}
% \label{eq-boundary-projection}
%     x = p(x) + d(x) \nu(p(x)),
% \end{equation}
% i.e.\ defining the projection $x\mapsto p(x)$.
In this subsection we describe the setting and recall some approximation results from \cite{Dziuk88,DziukElliott_ESFEM,ElliottRanner}.

The smooth domain $\Om$ is approximated by a polygonal domain $\Om_h$ with boundary surface $\Ga_h:=\pa \Om_h$, triangulated with a mesh size $h$ that is assumed sufficiently small in all the following. In particular, we require that for every point $x\in\Ga_h$ there is a unique point $p\in\Ga$ such that $x-p$ is orthogonal to the tangent space $T_p\Ga$ of $\Ga$ at $p$ (see Figure~\ref{fig:Omega_h}).
We  assume that the vertices of $\Ga_h$ are  on $\Ga$, i.e., $\Ga_h$ is an interpolation of $\Ga$.
We consider a family of quasi-uniform triangulations $\calT_h$, parametrised by the maximal meshwidth $h$.

Following \cite{Dziuk88}, we define a {\it lift} of functions $v_h:\Ga_h\to \R$ to $v_h^l:\Ga\to\R$ by setting $v_h^l(p)=v_h(x)$ for $p\in\Gamma$, where $x\in\Ga_h$ is the unique point on $\Ga_h$ with $x-p$  orthogonal to the tangent space $T_p\Ga$.
As in \cite{ElliottRanner}, we define a {\it lift} of functions $v_h:\Om_h\to \R$ to $v_h^l:\Om\to\R$ by setting $v_h^l(p)=v_h(x)$  if $x\in\Omega_h$ and $p\in \Omega$ are related as in Figure~\ref{fig:Omega_h}; see \cite{ElliottRanner} for a precise formal definition. Note that both definitions coincide on $\Gamma$.

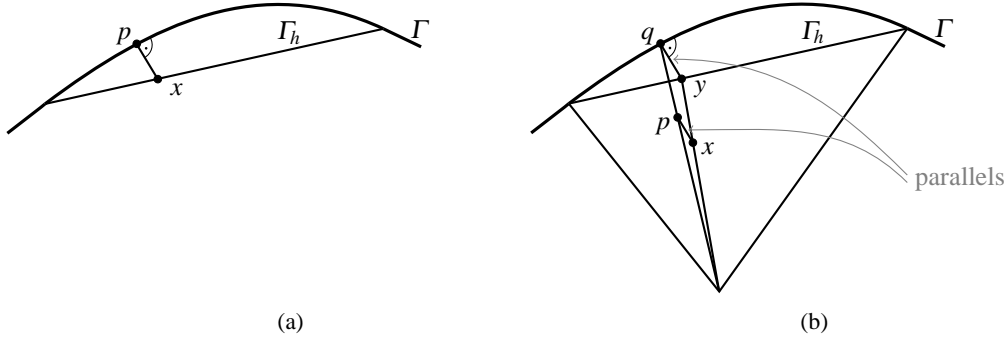
\begin{figure}
\begin{center}
\begin{subfigure}[b]{0.45\textwidth}
\begin{tikzpicture}
    % boundarys \Ga and \Ga_h
    \draw[very thick] (0,4) to [out=37,in=155] (4.5,5);
    \draw[thick] (0,4) -- (4.5,5);
    \draw[very thick] (0,4) to [out=37,in=38] (-0.5,3.61);
    \draw[very thick] (4.5,5) to [out=155,in=154] (5,4.76);
    \node at (3.25,4.95) {$\Ga_h$};
    \node at (5,5) {$\Ga$};
    % white "triangle"
    \draw[white] (2,2) -- (2,1.5);
    % point and its projection
    \draw [fill] (1.5,4.33) circle [radius=0.05];
    \node at (1.75,4.2) {$x$};
    \draw[thick] (1.5,4.33) -- (1.21,4.81);
    \draw [fill] (1.22,4.8) circle [radius=0.05];
    \node at (1.05,4.9) {$p$};
    % right angle mark
    \draw [fill] (1.345,4.74) circle [radius=0.01];
    \draw (1.4,4.88) to [out=300,in=30] (1.3260,4.6180);
\end{tikzpicture}
\caption{}
\end{subfigure}
\quad
\begin{subfigure}[b]{0.45\textwidth}
\begin{tikzpicture}
    % boundarys \Ga and \Ga_h
    \draw[very thick] (0,4) to [out=37,in=155] (4.5,5);
    \draw[thick] (0,4) -- (4.5,5);
    \draw[very thick] (0,4) to [out=37,in=38] (-0.5,3.61);
    \draw[very thick] (4.5,5) to [out=155,in=154] (5,4.76);
    \node at (3.25,4.95) {$\Ga_h$};
    \node at (5,5) {$\Ga$};
    % triangle
    \draw[thick] (0,4) -- (2,1.5) -- (4.5,5);
    % point and its projection
    %   boundary
    \draw [fill] (1.5,4.33) circle [radius=0.05];
    \node at (1.75,4.2) {$y$};
    \draw[thick] (1.5,4.33) -- (1.21,4.81);
    \draw [fill] (1.22,4.8) circle [radius=0.05];
    \node at (1.05,4.9) {$q$};
    %   bulk
    \draw[thick] (1.5,4.33) -- (2,1.5); % xo
    \draw[thick] (2,1.5) -- (1.21,4.81); % op
    \draw [fill] (1.65,3.481) circle [radius=0.05]; % original point
    \node at (1.85,3.4) {$x$};
    \draw [fill] (1.447,3.817) circle [radius=0.05]; % lifted point
    \node at (1.25,3.7) {$p$};
    \draw[thick] (1.65,3.481) -- (1.447,3.817); % xp
    % right angle mark
    \draw [fill] (1.345,4.74) circle [radius=0.01];
    \draw (1.4,4.88) to [out=300,in=30] (1.3260,4.6180);
    % arrows
    \draw [->,thin, gray] (4.5,3.05) to [out=135,in=0] (1.39,4.59);
    \draw [->,thin, gray] (4.5,2.95) to [out=135,in=0] (1.59,3.65);
    \node [gray] at (5.2,3) {parallels};
\end{tikzpicture}
\caption{}
\end{subfigure}
\end{center}
\caption{(a) shows the lift for surface functions: $v_h^l(p)=v_h(x)$; (b) shows the lift for bulk functions: $v_h^l(p)=v_h(x)$.}
\label{fig:Omega_h}
\end{figure}

The finite element space $V_h\nsubseteq H^1(\Om)$ corresponding to $\calT_h$ is spanned by continuous, piecewise linear nodal basis functions on $\Omega_h$, as in Section \ref{subsection-FEM}. We note here that  the restrictions of the basis functions to the boundary $\Ga_h$ again form a basis over the approximate boundary elements. We denote by
$V_h^l=\{ v_h^l\,:\, v_h\in V_h\}\subset V$ the lifted finite element space.

\begin{lemma}%[Interpolation results, \cite{Bernardi,Dziuk88} ]%and \cite{ElliottRanner}]
\label{lemma-nonpoly-interpolation}
    For $v\in H^2(\Om)$, such that $\gamma v \in H^2(\Ga)$, we denote by $I_h v\in V_h^l$ the lift of the finite element interpolation $\widetilde{I}_h v\in V_h$. Then the following estimates hold:
    \begin{enumerate}
      \item[(i)] Interpolation error in the bulk; see \cite{Bernardi,ElliottRanner}:
        \begin{equation*}
            \|v - I_h v\|_{L^2(\Om)} + h \|\nb(v - I_h v)\|_{L^2(\Om)} \leq C h^2 \|v\|_{H^2(\Om)}.
        \end{equation*}
      \item[(ii)] Interpolation error on the surface; see \cite{Dziuk88}:
        \begin{equation*}
            \|\gamma(v - I_h v)\|_{L^2(\Ga)} + h \|\nb_\Ga(v - I_h v)\|_{L^2(\Ga)}
            \leq Ch^2 \|\gamma v\|_{H^2(\Ga)}.
        \end{equation*}
    \end{enumerate}
\end{lemma}

\medskip
We introduce the discrete counterparts of the bilinear forms \eqref{a-heat-dynbc} and \eqref{b-heat-dynbc}, respectively:
\begin{alignat}{3}
    \label{a-heat-dynbc-nonpoly}
    a_h(u_h,v_h) =&\ \int_{\Om_h} \nb u_h \cdot \nb v_h \, \d x + \KAPPA \int_{\Ga_h} (\gamma_h u_h)(\gamma_h v_h)\, \d\sigma_h
+\BETA \int_{\Ga_h} \nb_{\Ga_h}u_h \cdot \nb_{\Ga_h}v_h\, \d\sigma_h &\ \qquad &\ (\KAPPA>0,\,\BETA\ge 0),\\
    \label{b-heat-dynbc-nonpoly}
    m_h(u_h,v_h) =&\ \int_{\Om_h} u_hv_h\, \d x + \MU  \int_{\Ga_h} (\gamma_h u_h)(\gamma_h v_h)\, \d\sigma_h &\ \qquad &\ (\MU> 0),
\end{alignat}
where $\gamma_h$ denotes the trace operator onto $\Ga_h$, and $\nb_{\Ga_h}$ is the discrete tangential gradient.

Using the lift and its properties (see Section 5 in \cite{DziukElliott_L2} and Section 4.3 in \cite{ElliottRanner} for the surface and the bulk, \resp), the following estimate of the geometric errors is obtained.
\begin{lemma}
\label{lemma-geometric-error}
    For any $v_h,w_h \in V_h$ we have the estimates
    \begin{align*}
     &   |a(v_h^l,w_h^l) - a_h(v_h,w_h)| \leq \ Ch \|\nb v_h^l\|_{L^2(B_h^l)} \, \|\nb w_h^l\|_{L^2(B_h^l)}
        \\ &\qquad + C h^2 \bigg(\|\nb v_h^l\|_{H^1(\Om)} \, \|\nb w_h^l\|_{L^2(\Om)} + \BETA \|\nb_\Ga v_h^l\|_{L^2(\Ga)} \, \|\nb_\Ga w_h^l\|_{L^2(\Ga)} + \KAPPA \|\gamma v_h^l\|_{L^2(\Ga)} \, \|\gamma w_h^l\|_{L^2(\Ga)}\bigg),
        \\
              & |m(v_h^l,w_h^l) - m_h(v_h,w_h)| \leq \ Ch \|v_h^l\|_{L^2(B_h^l)} \, \|w_h^l\|_{L^2(B_h^l)}
               \\
               &\qquad + C h^2 \Big(\|v_h^l\|_{L^2(\Om)} \, \|w_h^l\|_{L^2(\Om)} +\MU \|\gamma v_h^l\|_{L^2(\Ga)} \, \|\gamma w_h^l\|_{L^2(\Ga)}\Big),
    \end{align*}
    where $B_h^l$ denotes the layer of lifted elements which have a boundary face.
\end{lemma}
This result can be shown in the same way as Lemma 6.1 in \cite{ElliottRanner}.  Lemma~\ref{lemma-geometric-error} is the extension of Lemma 5.5 in \cite{DziukElliott_L2} to the surface-bulk case. As a consequence of this lemma we also have the $h$-uniform equivalence of the norms $\|v_h^l\| \sim \|v_h\|_h$ and $|v_h^l | \sim |v_h |_h$.

The following  lemma provides a key estimate.
\begin{lemma}[\cite{ElliottRanner}, Lemma 6.3]
\label{lemma-layer-estimate}
    For all $v\in H^1(\Om)$ the following estimate holds:
    \begin{equation}
	\|v\|_{L^2(B_h^l)} \leq C h^\Half \|v\|_{H^1(\Om)}.
    \end{equation}
\end{lemma}

\subsection{The Ritz map for problems with time-independent coefficients}
The semi-discretisation of the parabolic problem with dynamic boundary conditions \eqref{heat-dynbc}, with a non-polygonal domain $\Om$, determines $u_h:[0,T]\to V_h$ such that
\begin{equation}
\label{heat-dynbc-semi-discrete-nonpoly}
    \begin{alignedat}{1}
        m_h(\dot u_h(t),v_h) + a_h(u_h(t),v_h) &= m_h(f(t),v_h) \qquad \forall \,v_h\in V_h \qquad (0<t\le T)\\
        m_h(u_h(0),v_h) &= m_h(u_{0} ,v_h) \qquad\ \ \, \forall \,v_h\in V_h .
    \end{alignedat}
\end{equation}
The Ritz map $R_h:V\to V_h^l$ is defined by first determining $\widetilde R_h u\in V_h$ for $u\in V$ via
\begin{equation}\label{ritz-nonpoly}
a_h(\widetilde{R}_h u, v_h) = a(u,v_h^l) \qquad \forall v_h \in V_h
\end{equation}
and then setting  $R_h u := (\widetilde{R}_h u)^l \in V_h^l$.

The Galerkin orthogonality does not hold in this situation, just up to a small defect. The following estimate is obtained from Lemma \ref{lemma-geometric-error}:
\begin{align}
    a(u-R_h u, v_h^l) =&\ a(u,v_h^l) - a(R_h u,v_h^l) = a_h(\widetilde{R}_h u,v_h) - a((\widetilde R_h u)^l,v_h^l) \nonumber \\
    \leq &\ Ch \|\nb R_h u\|_{L^2(B_h^l)} \, \|\nb v_h^l\|_{L^2(B_h^l)} + C h^2 \bigg(\|\nb R_h u\|_{H^1(\Om)} \, \|\nb v_h^l\|_{L^2(\Om)} \nonumber \\
    \label{Galerkin-orthog-defect}
    &\ \qquad \qquad + \BETA \|\nb_\Ga R_h u\|_{L^2(\Ga)} \, \|\nb_\Ga v_h^l\|_{L^2(\Ga)} + \KAPPA \|\gamma R_h u\|_{L^2(\Ga)} \, \|\gamma v_h^l\|_{L^2(\Ga)}\bigg).
\end{align}
We prove the following error estimate for the  Ritz map. We again set $\|v\|^2 = a(v,v)$.
\begin{lemma}
\label{lemma-Ritz-error-nonpoly}
    The error of the Ritz map for the elliptic bilinear form \eqref{a-heat-dynbc} on a smooth domain satisfies the first-order error bound in the energy norm
    \begin{equation*}
	\|u - R_h u\|^2 \leq Ch^2 \bigg(\|u\|_{H^2(\Om)}^2 + \BETA \|\gamma u\|_{H^2(\Ga)}^2\bigg),
    \end{equation*}
    where the constant $C$ is independent of $h\le h_0$ ($h_0$ sufficiently small) and $u\in H^2(\Om)$ with $\sqrt{\BETA}\gamma u \in H^2(\Ga)$.
\end{lemma}

\begin{proof}
Using the bound \eqref{Galerkin-orthog-defect}, we start by estimating as
\begin{align*}
    \|u - R_h u\|^2 =&\ a(u - R_h u , u - I_h u) + a(u - R_h u , I_h u - R_h u) \\
    \leq &\ \|u - R_h u\| \, \|u - I_h u\| \\
    &\ + Ch \|\nb R_h u\|_{L^2(B_h^l)} \, \|\nb (I_h u - R_h u)\|_{L^2(B_h^l)} + Ch^2 \|R_h u\| \ \|I_h u - R_h u\|.
\end{align*}
These three terms need to be estimated in a suitable and careful way.
We use the interpolation estimates from Lemma \ref{lemma-nonpoly-interpolation} to obtain for the first term
\begin{align*}
    \|u - R_h u\| \, \|u - I_h u\| \leq &\ \Half \|u - R_h u\|^2 + Ch^2 \bigg(\|u\|_{H^2(\Om)}^2 + \BETA \|\gamma u\|_{H^2(\Ga)}^2 \bigg).
\end{align*}
For the second term we have
\begin{align*}
    &\ Ch \|\nb R_h u\|_{L^2(B_h^l)} \, \|\nb (I_h u - R_h u)\|_{L^2(B_h^l)} \\
    &\leq \ Ch \Big(\|\nb (u - R_h u)\|_{L^2(B_h^l)} + \|\nb u\|_{L^2(B_h^l)}\Big) \ \Big(\|\nb (u - I_h u)\|_{L^2(B_h^l)} + \|\nb (u - R_h u)\|_{L^2(B_h^l)}\Big) \\
    &\leq \ Ch \Big(2\|\nb (u - R_h u)\|_{L^2(B_h^l)}^2 + \|\nb u\|_{L^2(B_h^l)}^2 + \|\nb (u - I_h u)\|_{L^2(B_h^l)}^2\Big) \\
    &\leq \ Ch \Big(2\|u - R_h u\|^2 + Ch \|u\|_{H^2(\Om)}^2 + Ch^2 \|u\|_{H^2(\Om)}^2\Big),
\end{align*}
where in the last estimate we used Lemma \ref{lemma-layer-estimate} and the interpolation estimate.

The last term is estimated, using Young's inequality and the interpolation results of Lemma \ref{lemma-nonpoly-interpolation}, as
\begin{align*}
    &\ Ch^2 \|R_h u\| \, \|I_h u - R_h u\| \\
    &\leq\ Ch^2 \Big(\|u - R_h u\| + \|u\|\Big) \Big(\|u - I_h u\| + \|u - R_h u\|\Big) \\
    &\leq \ Ch^2 \Big(2\|u - R_h u\|^2 + \|u\|^2 + \|u - I_h u\|^2\Big) \\
    &\leq \ Ch^2 \Big(2\|u - R_h u\|^2 + C (\|u\|_{H^1(\Om)}^2 + \BETA \|\gamma u\|_{H^1(\Ga)}^2) + Ch^2\big(\|u\|_{H^2(\Om)}^2 + \BETA \|\gamma u\|_{H^2(\Ga)}^2\big) \Big).
\end{align*}
Reinserting all these into the estimate above, absorbing the terms $\|u - R_h u\|^2$ using an $h\leq h_0$ with a sufficiently small $h_0$, we obtain the stated result.
% \begin{align*}
%     \|u - R_h u\|^2 \leq &\ Ch^2 \bigg(\|u\|_{H^2(\Om)}^2 + \BETA \|\gamma u\|_{H^2(\Ga)}^2\bigg).
% \end{align*}
% Therefore the assertion is shown.
\end{proof}

The second-order error bounds of the Ritz map can be shown analogously as in Section \ref{sect:spacediscretisation}. Note that Conditions \ref{condition-regularity-0} and \ref{condition-regularity-beta} are satisfied for smooth domains $\Om$.
\begin{lemma}
    The second-order estimates of Lemmas \ref{lemma-Ritz-error-0} and \ref{lemma-Ritz-error-beta} hold for smooth domains, for $h\leq h_0$ with a sufficiently small $h_0$.
\end{lemma}

\subsection{The Ritz map for problems with time-varying coefficients}

In the time-varying non-polygonal case the Ritz map $R_h(t) : V \rightarrow V_h^l$ with respect to the bilinear form \eqref{a-heat-dynbc-t}, for $0\leq t \leq T$, is defined via $\widetilde{R}_h(t) u \in V_h$ defined by
\begin{equation}
\label{eq-Ritz-nonpoly-t}
    a_h(t;\widetilde{R}_h(t) u, v_h) = a(t;u,v_h^l) \qquad \forall v_h \in V_h,
\end{equation}
and setting $R_h(t) u := (\widetilde{R}_h(t) u)^l \in V_h^l$.

\begin{lemma}
    The error bounds for the Ritz map and its time derivative given in Lemmas \ref{lemma-Ritz-t} and \ref{lemma-Ritz-dot} also hold in the case of non-autonomous problems and a smooth domain $\Om$, for $h\leq h_0$ with a sufficiently small $h_0$.
\end{lemma}
\begin{proof}
In order to prove the error estimates of the Ritz map in the non-polygonal time-varying case we have to make sure that the constants in the estimate of the geometric error $a_h(t;u_h,v_h) - a(t;u_h^l,v_h^l)$ (from Lemma \ref{lemma-geometric-error}) are independent of $t$. This holds because of the smoothness of the domain $\Om$, and since we assumed uniform bounds on the coefficient functions.
Therefore the first order error bound is obtained with the proof of Lemma~\ref{lemma-Ritz-error-nonpoly}.

The error of the time derivative of the Ritz map is shown similarly as in the proof of Lemma \ref{lemma-Ritz-dot}. We start by differentiating the definition of the Ritz map \eqref{eq-Ritz-nonpoly-t}, and obtain
\begin{align*}
    a\Big(t; \dot u\t , v_h^l\Big) - a_h\Big(t; \diff \big(\widetilde{R}_h\t u\t\big) , v_h\Big) = -\bigg( \pa_t a_h\Big(t; \widetilde{R}_h\t u\t , v_h\Big) - \pa_t a\Big(t; u\t , v_h^l\Big) \bigg).
\end{align*}
Using  $\big(\diff v_h\big)^l = \diff v_h^l$,  we estimate
\begin{align*}
    \Big\|\dot u - \diff (R_h u)\Big\|_t^2 = &\ a\Big(t; \dot u - \diff (R_h u) , \dot u - R_h \dot u \Big) + a\Big(t; \dot u - \diff (R_h u) , R_h \dot u - \diff (R_h u) \Big) \\
    \leq &\  a\Big(t; \dot u - \diff (R_h u) , \dot u - R_h \dot u\Big) - \pa_t a\Big(t; u - R_h u , R_h \dot u - \diff (R_h u) \Big) \\
    &\ + \bigg| a_h\Big(t; \diff \big(\widetilde{R}_h u\big) , \widetilde{R}_h \dot u - \diff (\widetilde{R}_h u) \Big) - a\Big(t; \diff \big(R_h u\big) , R_h \dot u - \diff (R_h u) \Big) \bigg| \\
    &\ + \bigg| \pa_t a_h\Big(t; \widetilde{R}_h u , \widetilde{R}_h \dot u - \diff (\widetilde{R}_h u) \Big) - \pa_t a\Big(t; R_h u , R_h \dot u - \diff (R_h u) \Big) \bigg|.
\end{align*}
The first two terms are estimated analogously in the proof of Lemma \ref{lemma-Ritz-dot}. The other two terms are estimated separately, like in the proof of Lemma \ref{lemma-Ritz-error-nonpoly}, but here  we use the error estimates for the Ritz map instead of the interpolation estimates.
For the third term we have by the time-varying version of Lemma~\ref{lemma-geometric-error}
\begin{align*}
     &\ \bigg| a_h\Big(t; \diff \big(\widetilde{R}_h u\big) , \widetilde{R}_h \dot u - \diff (\widetilde{R}_h u) \Big) - a\Big(t; \diff \big(R_h u\big) , R_h \dot u - \diff (R_h u) \Big) \bigg| \\
     &\leq \ Ch \Big( 2 \big\| \dot u - \diff (R_h u) \big\|_t^2 + Ch \|\dot u\|_{H^2(\Om)}^2 + Ch^2  \|\dot u\|_{H^2(\Om)}^2 \Big) \\
     &\ \ + Ch^2 \Big( 2 \big\| \dot u - \diff (R_h u) \big\|_t^2 + C \big( \|\dot u\|_{H^2(\Om)}^2 +  \|\sqrt{\BETA} \gamma \dot u\|_{H^2(\Ga)}^2 \big) + Ch^2 \big( \|\dot u\|_{H^2(\Om)}^2 +  \|\sqrt{\BETA} \gamma \dot u\|_{H^2(\Ga)}^2 \big) \Big).
\end{align*}
For the fourth term we note
\begin{equation*}
    \pa_t a(t;v,w) = \int_\Ga \dot \BETA \,\nb_\Ga v \cdot \nb_\Ga w \d \sigma + \int_\Ga \dot \KAPPA \,(\gamma v) (\gamma w) \d \sigma,
\end{equation*}
which only contains boundary terms with bounded coefficient functions. As in Lemma \ref{lemma-geometric-error} we obtain the bound
\begin{equation*}
    |\pa_t a(t;v_h^l,w_h^l) - \pa_t a_h(t;v_h,w_h)| \leq  C h^2 \bigg(  \|\sqrt{\BETA}\nb_\Ga v_h^l\|_{L^2(\Ga)} \, \|\sqrt{\BETA}\nb_\Ga w_h^l\|_{L^2(\Ga)} +  \|\sqrt{\KAPPA}\gamma v_h^l\|_{L^2(\Ga)} \, \|\sqrt{\KAPPA}\gamma w_h^l\|_{L^2(\Ga)}\bigg).
\end{equation*}
Therefore we have for the last term, similarly as for the previous one,
\begin{align*}
    &\ \bigg| \pa_t a_h\Big(t; \widetilde{R}_h u , \widetilde{R}_h \dot u - \diff (\widetilde{R}_h u) \Big) - \pa_t a\Big(t; R_h u , R_h \dot u - \diff (R_h u) \Big) \bigg| \\
    \leq &\ Ch^2 \big\| R_h u \big\|_t \ \big\| R_h \dot u - \diff (R_h u) \big\|_t \\
    \leq &\ Ch^2 \Big( \big\| \dot u - \diff (R_h u) \big\|_t^2 + C \big( \|\dot u\|_{H^2(\Om)}^2 +  \|\sqrt{\BETA} \gamma \dot u\|_{H^2(\Ga)}^2 \big) + Ch^2 \big( \|\dot u\|_{H^2(\Om)}^2 +  \|\sqrt{\BETA} \gamma \dot u\|_{H^2(\Ga)}^2 \big) \Big).
\end{align*}
Absorbing the terms $Ch  \big\| \dot u - \diff (R_h u) \big\|_t^2$ with  a sufficiently small $h$ in the left-hand term, the result follows.
\end{proof}

\subsection{Error bounds of the semi-discretisation}

With the obtained error bounds for the Ritz map and the geometric error bounds, all the proofs of error bounds given in Section~\ref{sect:spacediscretisation} now extend directly to the case of smooth domains. We summarize this in the following theorem.

\begin{theorem}
    The semi-discrete error bounds of Theorems \ref{theorem-spatial-error-est-linear} --  \ref{theorem-spatial-error-est-semi-linear } also hold for smooth domains, for $h\leq h_0$ with a sufficiently small $h_0$.
\end{theorem}

\section{Time discretisation}

\subsection{Standard implicit integrators}
Since the parabolic problems with dynamic boundary conditions are cast in the same abstract setting of parabolic problems in which time integration by backward difference formulae (BDF) or algebraically stable implicit Runge--Kutta methods (such as the Radau IIA methods) has been analyzed before, results on the semi-discretisation in time by these methods can be applied directly; see, e.g., \cite{AkrivisLubich_quasilinBDF,Crouzeix1975,LubichOstermann_RK,Savare} for stability and error analyses for linear, semi-linear  and quasi-linear problems in an abstract setting that applies also to the problems considered here. Together with error bounds of the defects obtained on inserting the Ritz projection of the exact solution into the spatial semi-discretisation, such as derived in Sections 3 and 4, one then obtains error bounds for the full discretisation from the known results.

As an illustration of this procedure, we consider the $k$-step BDF time discretisation / finite element space discretisation of the linear non-autonomous problem \eqref{heat-dynbc-t}: with the bilinear forms
\eqref{a-heat-dynbc-t} and \eqref{b-heat-dynbc-t}, the finite element space $V_h$, a time stepsize $\tau>0$ and given starting values $u_h^0,\dots,u_h^{k-1}\in V_h$, we determine
$u_h^n\in V_h$ for $n\ge k$ with $t^n=n\tau\le T$ from the equation
\begin{equation}
\label{heat-dynbc-t-discrete}
        m(t^n;\partial_t^\tau u_h^n,v_h) + a(t^n;u_h^n,v_h) = m(t^n;f(t^n),v_h) \qquad \forall \,v_h\in V_h ,
\end{equation}
where $\partial_t^\tau u_h^n$ denotes the $k$-th order backward difference approximation to the time derivative,
$$
\partial_t^\tau u_h^n = \frac1\tau \sum_{j=0}^k \delta_j u_h^{n-j}
$$
with the coefficients $\delta_j$ given by the generating polynomial $\sum_{j=0}^k \delta_j\zeta^j = \sum_{\ell=1}^k \frac 1\ell (1-\zeta)^\ell$. In matrix-vector form, this is the implicit scheme
$$
 \Bigl(\frac{\delta_0}\tau\bfM(t^n)+ \bfA(t^n)\Bigr) \bfu^n = \bfb(t^n) - \bfM(t^n)\frac1\tau\sum_{j=1}^k  \delta_j \bfu^{n-j}.
$$
We then have the following fully discrete analogue  of Theorem~\ref{theorem-spatial-error-est-nonauto}.
\begin{theorem}
\label{theorem-spatial-error-est-nonauto-bdf}
    If the solution of the linear non-autonomous parabolic equation with dynamic boundary condition  \eqref{heat-dynbc-t} is sufficiently regular and if the starting values are sufficiently accurate, then the error of the  full discretisation \eqref{heat-dynbc-t-discrete} with linear finite elements and the $k$-step BDF method with $k\le 5$ satisfies the following error bound:
    \begin{equation*}
        \begin{alignedat}{1}
            & \|u_h^n - u(t^n)\|_{L_2(\Om)}^2 +  \int_\Ga \MU(.,t^n)\Big(\gamma u_h^n - \gamma u(t^n)\Big)^2\d \sigma + \tau\sum_{j=k}^n \big\|\nb\big(u_h^j - u(t^j)\big)\big\|_{L_2(\Om)}^2  \\
            & + \tau\sum_{j=k}^n \int_\Ga \BETA(.,t^j)\bigl(\nb_\Ga u_h^j - \nb_\Ga u(t^j)\bigr)^2 \d \sigma  + \tau\sum_{j=k}^n \int_\Ga \KAPPA(.,t^j)\bigl(\gamma u_h^j - \gamma u(t^j)\bigr)^2 \d \sigma  \leq C (h+\tau^{k})^2,
        \end{alignedat}
    \end{equation*}
for $n\ge k$ and $t^n=n\tau\leq T$, where the constant $C$ is independent of $h$, $\tau$ and $t$, but depends on $T$.% and $\overline{\KAPPA}$ and $\overline{\BETA}$.
\end{theorem}

\begin{proof}
 The error equation for $e_h^n=u_h^n-R_h(t^n)u(t^n)$ is
\begin{equation*}
    m(t^n;\partial_t^\tau e_h^n,v_h) + a(t^n;e_h^n,v_h) = m(t^n;d_h^n,v_h)
  \qquad \forall v_h\in V_h,
\end{equation*}
where
$$
d_h^n = \bigl(\dot u(t^n)-(\partial_t^\tau u)(t^n)\bigr) + \partial_t^\tau\bigl(u-R_hu)(t^n) .
$$
Here the first term is bounded, for example in the norm $\|v\|_t^2 = a(t;v,v)$ or $|v|_t^2 = m(t;v,v)$, by $O(\tau^k)$ for a temporally smooth solution $u$. The second term  is a linear combination of $k$ shifted difference quotients
$$
\frac1\tau\Bigl((u-R_hu)(t^{n-j})-(u-R_hu)(t^{n-j-1})\Bigr) = \int_0^1 \frac d{dt} (u-R_hu)(t^{n-j-1}+\theta\tau)\, \d\theta.
$$
Using Lemma~\ref{lemma-Ritz-dot}, we thus obtain
$$
\|d_h^n\|_{t^n} \le C(\tau^k+h).
$$
The following stability estimate is proved in \cite{LubichMansourVenkataraman_bdsurf} by testing with $v_h=e_h^n-\eta e_h^{n-1}$ with $\eta=0$, 0, 0.0836, 0.2878, 0.8160 for $k=1,2,\dotsc,5$, respectively,  using results of \cite{Dahlquist,NevanlinnaOdeh} and the bounds \eqref{m-dot-bound} and \eqref{a-dot-bound}: for $t^n\le T$,
$$%\begin{equation}
|e_h^n|_{t^n}^2 + \tau \sum_{j=k}^{n} \|e_h^j\|_{t^j}^2
\le C\,\dt\sum_{j=k}^{n} \|d_h^j\|_{*,t^j}^2 + C \max_{0\le i\le k-1} |e_h^i|_{t^i}^2. %+ C\eta^2\tau \|\bfe_{k-1}\|_{t_{k-1}}^2,
$$%\end{equation}
Using Lemma~\ref{lemma-Ritz-t} and recalling the definition of the bilinear forms $a$ and $m$ of \eqref{a-heat-dynbc}--\eqref{b-heat-dynbc} then yields the result.
\end{proof}

We note that the order in $h$ increases from 1 to $m$  if finite elements of degree $m$ are used.

The result for the semi-linear  equation, Theorem~\ref{theorem-spatial-error-est-semi-linear }, is extended to the BDF full discretisation in the same way under the mild stepsize restriction $\tau^k \le C h^{d/2}$, which allows us to show the $L^\infty$ bound of the fully discrete solution  in the same way as in the semi-discrete case.

%
% \subsubsection{Implicit Runge--Kutta methods}
%
% Similarly as for quasi-linear problems \cite{LubichOstermann_RK} or for non-autonomous problems \cite{DziukLubichMansour}, as a temporal discretisation we consider an implicit Runge--Kutta method, which is algebraically stable, it has stage order $q$ and order $p\geq q+1$, having an invertible coefficient matrix and $|R(\infty)|<1$. Applying the techniques from \cite{LubichOstermann_RK} the following error estimates hold for the (lumped mass) finite element method
%
% \subsubsection{Backward difference formulae}
% ...

\subsection{Exponential integrators}

Exponential integrators have become attractive for use with parabolic problems in recent years since they allow for large time steps while using just matrix-vector multiplications; see the review by \cite{HoO}.
While it is not obvious from the outset how to use exponential integrators for parabolic problems with dynamic boundary conditions such as \eqref{heat-dynbc}, the above weak formulation and finite element discretisation with mass lumping yields the system
$$
\M \dot\bfu(t) + \bfA \bfu(t) = \bfb(t)
$$
with a diagonal positive definite mass matrix $\M$. With $\widehat\bfA = \M^{-1/2} \bfA \M^{-1/2}$ and $\widehat\bfb(t)=\M^{-1/2}\bfb(t)$ and $\bfy(t)=\M^{1/2}\bfu(t)$, the transformed system becomes
$$
\dot\bfy(t) = -\widehat\bfA \bfy(t)+ \bfb(t),
$$
to which an exponential integrator is readily applied. Here we consider for simplicity of presentation just the exponential Euler method
$$
\bfy^{n+1} = \bfy^n + \tau \varphi(-\tau\widehat \bfA) \bigl(-\widehat \bfA \bfy^n + \widehat\bfb(t^n) \bigr),
$$
with the entire function $\varphi(z)=(e^z-1)/z$. The matrix-vector product $\varphi(-\tau\widehat \bfA)\bfv$ can be approximated efficiently by Krylov subspace methods; see \cite{HoL,Sa}. Since the largest eigenvalues of the symmetric positive semi-definite matrix $\widehat\bfA$ are of size $O(h^{-2})$ in the case of a quasi-uniform triangulation with mesh size $h$, the error bounds of \cite{HoL} show that using $m$ Krylov substeps, each of which requires one multiplication of $\widehat\bfA$ with a vector, with a step size $\tau=O(m^2h^2/|\log\varepsilon|)$ the matrix-vector product $\varphi(-\tau\widehat \bfA)\bfv$ is approximated with an $O(\varepsilon)$ error. For large $m\sim h^{-1}$ the effective step size $\tau_{\rm eff} =\tau/m$ can thus be of size $O(h)$, as opposed to only $O(h^2)$ for the explicit Euler method or standard explicit Runge--Kutta methods. Runge--Kutta--Chebyshev methods, such as those in \cite{Abdulle2002}, achieve the same increase in the effective  step size.
%Since Krylov methods, in contrast to Chebyshev methods, can benefit from a clustering of eigenvalues, they still work well for a large coefficient $\KAPPA$ in solving problem \eqref{heat-dynbc}.

\subsection{Bulk-surface splitting integrators}
Especially in situations with different time scales in the domain and on the boundary, such as fast reaction-diffusion on the surface and slow diffusion in the bulk, it appears attractive to use splitting methods that treat the problem in the bulk and on the surface separately. There are basically two ways in which such a splitting can be done:
\begin{itemize}
 \item {\it Bulk-surface force splitting}\/: Write the stiffness matrix as $\bfA=\bfA_\Omega + \bfA_\Gamma$, where
 $\bfA_\Omega$ and $\bfA_\Gamma$ are the stiffness matrices corresponding to the bilinear forms $a_\Omega(\cdot,\cdot)$  and $a_\Gamma(\cdot,\cdot)$ defined by the bulk and surface integrals, respectively. We use a corresponding decomposition for the load vector $\bfb(t)=\bfb_\Omega(t)+\bfb_\Gamma(t)$. Then a time step with the Strang splitting for this decomposition reads as follows:
 \begin{enumerate}
  \item Over half a time step solve $\M\dot\bfu(t) = -\bfA_\Gamma \bfu(t)+ \bfb_\Gamma(t_0)$ for $0\le t \le t_{1/2} = t_0+\tfrac12\tau$ with initial value $\bfu^0$, to obtain the final value $\bfu^{1/2,-}$.
  \item Over a full time step solve $\M\dot\bfu(t) = -\bfA_\Omega \bfu(t)+ \bfb_\Omega(t_{1/2})$ for $0\le t \le t_{1} = t_0+\tau$ with initial value $\bfu^{1/2,-}$, to obtain the final value $\bfu^{1/2,+}$.
  \item Over half a time step solve $\M\dot\bfu(t) = -\bfA_\Gamma \bfu(t)+ \bfb_\Gamma(t_1)$ for $t_{1/2}\le t \le t_{1}$ with initial value $\bfu^{1/2,+}$, to obtain the final value $\bfu^1$.
 \end{enumerate}
With a diagonal mass matrix $\M$, the first and third substep only change nodal values that correspond to surface nodes.
In the second substep, $\bfb_\Omega(t_{1/2})$ can be replaced by $\tfrac12 (\bfb_\Omega(t_{0})+\bfb_\Omega(t_{1}))$.

 \item {\it Bulk-surface component splitting}\/: We write the nodal vector as
 $$
 \bfu=\begin{pmatrix} \bfu_0 \\ \bfu_1 \end{pmatrix},
 $$
 where $\bfu_0$ and $\bfu_1$ contain the approximate solution values at the inner nodes and boundary nodes, respectively. We write similarly
 $$
 \M = \begin{pmatrix}
         \M_0 & \bfzero \\
         \bfzero & \M_1
        \end{pmatrix},
        \quad\
 \bfA = \begin{pmatrix}
         \bfA_{00} & \bfA_{01}
         \\
         \bfA_{10} & \bfA_{11}
        \end{pmatrix},
        \quad\
 \bfb = \begin{pmatrix} \bfb_0 \\ \bfb_1 \end{pmatrix}.
 $$
A time step with the Strang splitting for this decomposition reads as follows:
\begin{enumerate}
  \item Over half a time step solve $\M_1\dot\bfu_1(t) = -\bfA_{11} \bfu_1(t)-\bfA_{10} \bfu_0^0+ \bfb_1(t_0)$ for $0\le t \le t_{1/2}$ with initial value $\bfu_1^0$, to obtain the final value $\bfu_1^{1/2}$.
  \item Over a full time step solve $\M_0\dot\bfu_0(t) = -\bfA_{00} \bfu_0(t)-\bfA_{01} \bfu_0^{1/2}+ \bfb_0(t_{1/2})$ for $0\le t \le t_{1}$ with initial value $\bfu_0^0$, to obtain the final value $\bfu_0^1$.
  \item Over half a time step solve $\M_1\dot\bfu_1(t) = -\bfA_{11} \bfu_1(t)-\bfA_{10} \bfu_0^1+ \bfb_1(t_1)$ for $t_{1/2}\le t \le t_{1}$ with initial value $\bfu_1^{1/2}$, to obtain the final value $\bfu_1^1$.
\end{enumerate}
In the second substep, $\bfb_0(t_{1/2})$ can be replaced by $\tfrac12 (\bfb_0(t_{0})+\bfb_0(t_{1}))$.
\end{itemize}

For both types of splittings, the local error (that is, the error after one step starting from the exact solution) is of order $O(\tau^3)$ under sufficient regularity assumptions on the solution; cf., e.g., \cite{DesS02,HanO09,JahL00} for different techniques to bound the local error of the Strang splitting. This yields an $O(\tau^2)$ global error bound uniformly over bounded time intervals {\it in a norm in which the method is stable}.

The force splitting is $L_2(\Omega)\oplus L_2(\Gamma)$-stable: with the symmetric positive definite matrices $\widehat\bfA_\Omega = \M^{-1/2} \bfA_\Omega \M^{-1/2}$ and $\widehat\bfA_\Gamma = \M^{-1/2} \bfA_\Gamma \M^{-1/2}$ we have (taking here $\bfb(t)=\bfzero$ for simplicity),
$$
\M^{1/2} \bfu^1 = e^{-\frac\tau2 \widehat \bfA_\Gamma} \, e^{-\tau \widehat \bfA_\Omega} \,
e^{-\frac\tau2 \widehat \bfA_\Gamma}\,\M^{1/2}\bfu^0
$$
and hence
$$
(\bfu^1)^T \M \bfu^1 \le (\bfu^0)^T \M \bfu^0,
$$
where we note the $h$-uniform equivalence of the norms 
(see Section~\ref{subsubsec:mass-lumping})
$$(\bfu^T\M\bfu)^{1/2} = |u_h|_{\rm LM} \sim |u_h| =
\bigl( \|u_h\|_{L_2(\Omega)}^2 + \mu \|u_h\|_{L_2(\Gamma)}^2 \bigr)^{1/2}.
$$
%\|u_h^1\|_{L_2(\Omega)}^2 + \mu\|u_h^1\|_{L_2(\Gamma)}^2 = (\bfu^1)^T \M \bfu^1 \le (\bfu^0)^T \M \bfu^0
%= \|u_h^0\|_{L_2(\Omega)}^2 + \|u_h^0\|_{L_2(\Gamma)}^2.
%$$

Stability is not evident for the component splitting, since here we have (again for $\bfb(t)=\bfzero$) a composition of exponentials of non-symmetric matrices:
$$
\M^{1/2} \bfu^1 = \exp\left({-\frac\tau2 \begin{pmatrix} \bfzero & \bfzero \\
                                     \widehat \bfA_{10} & \widehat \bfA_{11}
                                    \end{pmatrix} }\right)\,
  \exp\left({-\tau \begin{pmatrix}
              \widehat \bfA_{00} & \widehat \bfA_{01}
              \\
              \bfzero & \bfzero
                                    \end{pmatrix}} \right)\,
  \exp\left({-\frac\tau2 \begin{pmatrix} \bfzero & \bfzero \\
                                     \widehat \bfA_{10} & \widehat \bfA_{11}
                                    \end{pmatrix}} \right)\,
                                    \M^{1/2}\bfu^0.
 $$
We have, however, the following stability bound in a stepsize-dependent norm related to the energy norm. Here, $\bfy^1=\M^{1/2}\bfu^1$ denotes the mass-scaled numerical solution after one step of the above component splitting method with $\bfb(t)=0$, and $\|\cdot\|_2$ denotes the Euclidean norm or its induced matrix norm.

\begin{lemma} \label{lem:stab-split}
For the matrix
$$
\bfL = \begin{pmatrix}
(\bfI-e^{-\tau \widehat\bfA_{00}})^{-1/2}  & \bfzero \\
\bfL_{10} & (\bfI-e^{-\tau \widehat\bfA_{11}})^{-1/2}
\end{pmatrix}
\begin{pmatrix}
\widehat\bfA_{00} ^{1/2} & \bfzero \\
\bfzero & \widehat\bfA_{11} ^{1/2}
\end{pmatrix}
$$
with the off-diagonal block $\bfL_{10}= (\bfI-e^{-\tau \widehat\bfA_{11}/2})^{1/2} (\bfI+e^{-\tau \widehat\bfA_{11}/2})^{-1/2} (\widehat\bfA_{11}^{-1/2} \widehat\bfA_{10} \widehat\bfA_{00} ^{-1/2})$, which has $\|\bfL_{10}\|_2\le 1$,
we have %for all $n\ge 1$,
$$
\| \bfL  \bfy^1 \|_2 \le \| \bfL  \bfy^0 \|_2.
$$
\end{lemma}

\begin{proof}
 For ease of notation we omit the hats on $\widehat\bfA_{00}$ etc.~in the following. We denote the exponential matrices
 \begin{eqnarray*}
  \bfE_0(\tau) &=& \exp\left({-\tau \begin{pmatrix}
               \bfA_{00} &  \bfA_{01}
              \\
              \bfzero & \bfzero
                                    \end{pmatrix}} \right) =
         \begin{pmatrix}
          e^{-\tau \bfA_{00}} & - (\bfI-e^{-\tau \bfA_{00}}) \bfA_{00}^{-1} \bfA_{01} \\
          \bfzero & \bfI
         \end{pmatrix}
         \\
  \bfE_1(\tau) &=&  \exp\left({-\tau \begin{pmatrix} \bfzero & \bfzero \\
                                      \bfA_{10} & \bfA_{11}
                                    \end{pmatrix}} \right)
                = \begin{pmatrix}
                   \bfI & \bfzero \\
                   -(\bfI-e^{-\tau \bfA_{11}})\bfA_{11}^{-1}\bfA_{10} & e^{-\tau \bfA_{11}}
                  \end{pmatrix}
  \end{eqnarray*}
so that the propagation matrix of the Lie-Trotter splitting is
$$
\bfS_{\rm Lie} = \bfE_0(\tau)\, \bfE_1(\tau)
$$
and that of the Strang splitting is
$$
\bfS_{\rm Strang} = \bfE_1(\tau/2)\,\bfE_0(\tau)\, \bfE_1(\tau/2) = \bfE_1(\tau/2)\, \bfS_{\rm Lie} \, \bfE_1(\tau/2)^{-1}.
$$
It turns out that the block diagonal matrix
$$
\bfT = \begin{pmatrix}
        (\bfI-e^{-\tau \bfA_{00}})^{1/2} \bfA_{00}^{-1/2} & \bfzero \\
        \bfzero & e^{\tau\bfA_{11}/2} (\bfI-e^{-\tau \bfA_{11}})^{1/2}\bfA_{11}^{-1/2}
       \end{pmatrix}
$$
symmetrizes the Lie matrix:
$$
\widetilde \bfS := \bfT^{-1} \bfS_{\rm Lie} \bfT
$$
equals the symmetric matrix, with the abbreviation $\bfX^T =  (\bfI-e^{-\tau \bfA_{00}})^{1/2} \bfA_{00}^{-1/2} \bfA_{01}
\bfA_{11}^{-1/2}$,
$$
\widetilde \bfS = \begin{pmatrix}
                 e^{-\tau \bfA_{00}} + \bfX^T (\bfI-e^{-\tau \bfA_{11}}) \bfX &
                 -\bfX^T (\bfI-e^{-\tau \bfA_{11}})^{1/2} e^{-\tau\bfA_{11}/2} \\
                 -e^{-\tau\bfA_{11}/2} (\bfI-e^{-\tau \bfA_{11}})^{1/2} \bfX &
                 e^{-\tau \bfA_{11}}
                \end{pmatrix}.
$$
We next show that $\| \widetilde\bfS \|_2 \le 1$. Since the matrix is symmetric, this is equivalent to showing that the numerical range $\{\bfv^T\widetilde\bfS\bfv\,:\, \|\bfv\|_2\le 1\}$ is contained in the interval $[-1,1]$. With $\bfv^T=(\bfv_0^T,\bfv_1^T)$ the calculation yields
\begin{eqnarray*}
\bfv^T\widetilde\bfS\bfv &=& \bfv_0^T e^{-\tau \bfA_{00}} \bfv_0 + \bfv_0^T \bfX^T (\bfI-e^{-\tau \bfA_{11}}) \bfX \bfv_0
\\
&& -2 \bfv_0^T \bfX^T e^{-\tau\bfA_{11}/2} (\bfI-e^{-\tau \bfA_{11}})^{1/2}\bfv_1 + \bfv_1^T e^{-\tau \bfA_{11}} \bfv_1,
\end{eqnarray*}
where the mixed term is estimated by
$$
-2 \bfv_0^T \bfX^T e^{-\tau\bfA_{11}/2}\cdot (\bfI-e^{-\tau \bfA_{11}})^{1/2}\bfv_1 \le
\bfv_0^T \bfX^T e^{-\tau\bfA_{11}}\bfX \bfv_0 + \bfv_1^T (\bfI-e^{-\tau \bfA_{11}}) \bfv_1,
$$
so that some terms cancel and we obtain
$$
\bfv^T\widetilde\bfS\bfv \le  \bfv_0^T e^{-\tau \bfA_{00}} \bfv_0 + \bfv_0^T \bfX^T\bfX \bfv_0 + \bfv_1^T\bfv_1.
$$
Since the Schur complement $\bfA_{00}- \bfA_{01}\bfA_{11}^{-1}\bfA_{01}^T$ is symmetric positive definite,
we infer that also the transformed matrix
$\bfI - \bfA_{00}^{-1/2}\bfA_{01}\bfA_{11}^{-1/2}\cdot \bfA_{11}^{-1/2}\bfA_{01}^T\bfA_{00}^{-1/2}$ is positive definite and hence $\|\bfA_{00}^{-1/2}\bfA_{01}\bfA_{11}^{-1/2}\|_2\le 1$, so that on recalling the definition of $\bfX$
we can bound
$$
\bfv_0^T \bfX^T\bfX \bfv_0 \le \bfv_0^T (\bfI-e^{-\tau \bfA_{00}}) \bfv_0.
$$
This gives us finally
$$
\bfv^T\widetilde\bfS\bfv \le \bfv_0^T \bfv_0 + \bfv_1^T\bfv_1 = \|\bfv\|_2^2
$$
and in the same way we also obtain
$$
\bfv^T\widetilde\bfS\bfv \ge - \|\bfv\|_2^2.
$$
Hence we have shown that $\| \widetilde\bfS \|_2 \le 1$.
Since
$$
\bfS_{\rm Strang}   = \bfE_1(\tau/2) \,\bfS_{\rm Lie}  \,\bfE_1(\tau/2)^{-1} =
\bfE_1(\tau/2)\,\bfT  \widetilde\bfS  \bfT^{-1} \,\bfE_1(\tau/2)^{-1},
$$
we obtain with $\bfL=   (\bfE_1(\tau/2)\bfT)^{-1}$, which is the matrix $\bfL$ given in the statement of the lemma,
that for $\bfy^1=\bfS_{\rm Strang}\bfy^0$ we have the bound
$$
\| \bfL \bfy^1 \|_2 = \| \widetilde \bfS \bfL \bfy^0 \|_2 \le \| \bfL \bfy^0 \|_2,
$$
which concludes the proof.
\end{proof}

\section*{Acknowledgement}
We thank Andr\'as B\'atkai for stimulating discussions in the beginnings of this work. The research stay of B.K.~at the University of T\"ubingen has been funded by DAAD.

%\clearpage
\bibliographystyle{IMANUM-BIB}
\bibliography{dynamic}
\end{document}